\documentclass[reqno,11pt]{amsart}
\usepackage{epsfig}
\usepackage{color}
\usepackage{enumerate} 
\usepackage{times}
\usepackage{amsfonts,amsmath,amssymb}
\usepackage{mathrsfs}
\usepackage[latin1]{inputenc}
\usepackage[T1]{fontenc}
\numberwithin{equation}{section}
\newtheorem{theorem}{Theorem}[section]
\newtheorem{lemma}[theorem]{Lemma}
\newtheorem{remark}[theorem]{Remark}

\newtheorem{Corollary}[theorem]{Corollary}
\allowdisplaybreaks

\newtheorem{Prop}{Proposition}[section]

\def\qed{\hfill$\square$\par \bigskip}

\newcommand{\R}{\mathbb{R}}

\newcommand{\D}{\mathcal{D}}

\newcommand{\s}{\mathbb{S}}

\newcommand{\nablas}{\nabla_{x'}}
\newcommand{\norm}[1]{\|#1\|}
\newcommand{\abs}[1]{\left\vert#1\right\vert}

\newcommand{\para}[1]{\left(#1\right)}

\newcommand{\seq}[1]{\left<#1\right>}

\def\beq{\begin{equation}}
\def\eeq{\end{equation}}
\renewcommand{\leq}{\leqslant}
\renewcommand{\geq}{\geqslant}
\newcommand{\bea}{\begin{eqnarray}}
\newcommand{\eea}{\end{eqnarray}}
\newcommand{\beas}{\begin{eqnarray*}}
\newcommand{\eeas}{\end{eqnarray*}}

{
\newcommand{\bel}{\begin{equation} \label}
\newcommand{\ee}{\end{equation}}

\newcommand{\pd}{\partial}
\newcommand{\cA}{\mathcal{A}}
\newcommand{\cB}{\mathcal{B}}
\newcommand{\sB}{\mathscr{B}}
\newcommand{\cC}{\mathcal{C}}
\newcommand{\cD}{\mathcal{D}}

\newcommand{\cH}{\mathcal{H}}
\newcommand{\sH}{\mathscr{H}}
\newcommand{\cN}{\mathcal{N}}

\newcommand{\cP}{\mathcal{P}}
\newcommand{\sX}{\mathscr{X}}
\newcommand{\cZ}{\mathcal{Z}}
\newcommand{\eps}{\epsilon}
\newcommand{\supp}{\mathrm{supp}}

\usepackage{graphicx}

\newcommand{\re}{\textrm{Re}}
\renewcommand{\d}{\textrm{d}}

\usepackage{times}
\begin{document}
\title{An inverse problem for the magnetic Schr\"odinger equation in infinite cylindrical domains}
\author{M. Bellassoued}
\address{University of Tunis El Manar, National Engineering School of Tunis, ENIT-LAMSIN, B.P. 37, 1002 Tunis, Tunisia}
\email{mourad.bellassoued@fsb.rnu.tn}
\author{Y. Kian}
\address{Aix-Marseille Universit\'e, CNRS, CPT UMR 7332, 13288 Marseille \& Universit\'e de Toulon, CNRS, CPT UMR 7332, 83957 La Garde, France.}
\email{yavar.kian@univ-amu.fr}
\author{E. Soccorsi}
 \address{Aix-Marseille Universit\'e, CNRS, CPT UMR 7332, 13288 Marseille \& Universit\'e de Toulon, CNRS, CPT UMR 7332, 83957 La Garde, France.}
\email{eric.soccorsi@univ-amu.fr}

\begin{abstract}
We study the inverse problem of determining the magnetic field and the electric potential entering the Schr\"odinger
equation in an infinite 3D cylindrical domain, by Dirichlet-to-Neumann map. The cylindrical domain we consider is a closed waveguide in the sense that the cross section is a bounded domain of the plane. We prove that the
knowledge of the Dirichlet-to-Neumann map determines uniquely, and even H\"older-stably, 
the magnetic field induced by the magnetic potential and the electric potential. Moreover, if the maximal strength of both the magnetic field and the electric potential, is attained in a fixed bounded subset of the domain, we extend the above results by taking finitely extended boundary observations of the solution, only.\medskip 
\end{abstract}

\maketitle
\section{Introduction}
\setcounter{equation}{0}

\subsection{Statement of the problem}
Let $\omega$ be a bounded and simply connected domain of $\R^2$ with $C^2$ boundary $\pd \omega$, and set $\Omega:=\omega\times\R$. 
For $T>0$, we consider the initial boundary value problem (IBVP)
\bel{1.1}
\left\{
\begin{array}{ll}
\para{i\pd_t+\Delta_A+q} u=0,  & \mbox{in}\ Q:=(0,T) \times \Omega,\\
u(0,\cdot)=0, & \mbox{in}\ \Omega,\\
u=f, & \mbox{on}\ \Sigma:=(0,T) \times \Gamma,
\end{array}
\right.
\ee
where $\Delta_A$ is the Laplace operator associated with the magnetic potential $A \in W^{1,\infty}(\Omega)^3$,
\bel{DelA}
\Delta_A:=\sum_{j=1}^3
\para{\pd_{x_j}+i a_j}^2=\Delta+2i A\cdot\nabla+i (\nabla \cdot A) -|A|^2,
\ee
and $q \in L^{\infty}(\Omega)$. We define the Dirichlet-to-Neumann (DN) map associated with \eqref{1.1}, as
\bel{DN}
\Lambda_{A,q}(f) :=(\pd_\nu+ i A \cdot \nu )u,\ f\in L^2(\Sigma),
\ee
where $\nu(x)$ denotes the unit outward normal to $\pd \Omega$ at $x$, and $u$ is the solution to \eqref{1.1}.

In the remaining part of this text, two magnetic potentials $A_j \in W^{1,\infty}(\Omega)^3$, $j=1,2$, are said gauge equivalent, if there exists $\Psi \in W^{2,\infty}(\Omega)$ obeying $\Psi|_{\Gamma}=0$, such that
\bel{gauge} 
A_2 = A_1 + \nabla \Psi.
\ee
In this paper we examine the uniqueness and stability issues in the inverse problem of determining 
the electric potential $q$ and the gauge class of $A$, from the knowledge of $\Lambda_{A,q}$.

\subsection{Physical motivations}

The system \eqref{1.1} describes the quantum motion of a charged particle (the various physical constants are taken equal to $1$) constrained by the unbounded domain $\Omega$, under the influence of the magnetic field generated by $A$, and the electric potential $q$. 
Carbon nanotubes whose length-to-diameter ratio is up to $10^8/1$, are commonly modeled by infinite waveguides such as $\Omega$.
In this context, the inverse problem under consideration in this paper can be rephrased as to whether the strength of the electromagnetic quantum disorder (namely, the magnetic field and the electric impurity potential $q$, see e.g. \cite{[CL],[KBF]}) can be determined by boundary measurement of the wave function $u$. 
\subsection{State of the art}

Inverse coefficients problems for partial differential equations such as the Schr\"odinger equation, are the source of challenging mathematical problems, and have attracted many attention over the last decades. For instance, using the Bukhgeim-Klibanov method (see \cite{[Bukhgeim-Klibanov], [Klibanov], [KT]}), \cite{[Baudouin-Puel]} claims Lipschitz stable determination of the time-independent electric potential perturbing the dynamic (i.e. non stationary) Schr\"odinger equation, from a single boundary measurement of the solution. In this case, the observation is performed on a sub boundary fulfilling the geometric optics condition for the observability, derived by Bardos, Lebeau and Rauch in \cite{[BLR]}. This geometrical condition was removed by \cite{[BC1]} for potentials which are {\it a priori} known in a neighborhood of the boundary, at the expense of weaker stability. In the same spirit, \cite{[CS]} Lipschitz stably determines by means of the Bughkgeim-Klibanov technique, the magnetic potential in the Coulomb gauge class, from a finite number of boundary measurements of the solution. Uniqueness results in inverse problems for the DN map related to the magnetic Schr\"odinger equation are also available in \cite{[Eskin]}, but they are based on a different approach involving geometric optics (GO) solutions. The stable recovery of the magnetic field by the DN map of the dynamic magnetic Schr\"odinger equation is established in \cite{[BC2]} by combining the approach used for determining the potential in hyperbolic equations (see \cite{[Bellassoued],[Bellassoued-Benjoud],[Bell-Jel-Yama1],[IS],[Rakesh-Symes],[SU],[Su1]}) with the one
employed for the idetification of the magnetic field in elliptic equations (see \cite{[DKSU],[Salo],[Su2]}).
Notice that in the one-dimensional case, \cite{[Avdonin-al]} proved by means of the boundary control method introduced by \cite{[Belishev]}, that the DN map uniquely determines the time-independent electric potential of the Schr\"odinger equation. 
In \cite{[BD]} the time-independent electric potential is stably determined by the DN map associated with the dynamic magnetic Schr\"odinger equation on a Riemannian manifold. This result was recently extended by \cite{[Bellassoued2]} to simultaneous determination of both the magnetic field and the electric potential. As for inverse coefficients problems of
the Schr\"odinger equation with either Neumann, spectral, or scattering data, we refer to \cite{[DKSU],[Eskin3],[EskinRal],[Ki1],[KLU],[KU],[Salo],[Salo1],[Su2],[Tzou]}.

All the above mentioned results are obtained in a bounded domain. Actually, there is only a small number of mathematical papers dealing with inverse coefficients problems in unbounded domains. One of them, \cite{[Rakesh1]}, examines the problem of determining a potential appearing in the wave equation in the half-space. Assuming that the potential is known outside a fixed compact set, the author proves that it is uniquely determined by the DN map. 
Unique determination of compactly supported potentials appearing in the stationary Schr\"odinger equation in an infinite slab from partial DN measurements is established in \cite{[LU]}. The same problem is addressed by \cite{[KLU]} for the stationary magnetic Schr\"odinger equation, and by \cite{[Y]} for bi-harmonic operators with perturbations of order zero or one. The inverse problem of determining the twisting function of an infinite twisted waveguide by the DN map, is addressed in \cite{[ChS]}. The analysis carried out in \cite{[IS],[Rakesh-Symes],[SU],[Su1]} is adapted to unbounded cylindrical domains in \cite{[ChS]} for the determination of time-independent potentials with prescribed behavior outside a compact set, by the hyperbolic DN map. In \cite{[KPS2]}, electric potentials with suitable exponential decay along the infinite direction of the waveguide, are stably recovered from a single boundary measurement of the solution. This is by means of a specifically designed Carleman estimate for the dynamic Schr\"odinger equation in infinite cylindrical domains, derived in \cite{[KPS1]}. The geometrical condition satisfied by the boundary data measurements in \cite{[KPS2]} is relaxed in \cite{[BKS]} for potentials which are known in a neighborhood of the boundary.
In \cite{[CKS1]}, time-dependent potentials that are periodic in the translational direction of the waveguide, are stably retrieved by the DN map of the Schr\"odinger equation. In \cite{[KKS]}, periodic potentials are stably retrieved from the asymptotics of the boundary spectral data of the Dirichlet Laplacian. Finally, we refer to \cite{[CKS2],[CKS3]}, for the analysis of the Calder\'on problem in a periodic waveguide.

\subsection{Well-posedness}
We start by examining the well-posedness of the IBVP \eqref{1.1} in the functional space $\mathcal{C}([0,T],H^1(\Omega)) \cap \mathcal{C}^1([0,T],H^{-1}(\Omega))$. 
Namely, we are aiming for sufficient conditions on the coefficients $A$, $q$ and the non-homogeneous Dirichlet data $f$, ensuring that \eqref{1.1} admits a unique solution in the transposition sense.
We say that $u \in L^{\infty}(0,T;H^{-1}(\Omega))$ is a solution to \eqref{1.1} in the transposition sense, if the identity
$$
\langle u, F \rangle_{L^\infty(0,T;H^{-1}(\Omega)),L^1(0,T;H^1_0(\Omega))}=\langle f , \pd_\nu v \rangle_{L^2(\Sigma)},
$$
holds for any $F \in L^1(0,T;H_0^1(\Omega))$. Here $v$ denotes the unique $\mathcal{C}([0,T],H^1(\Omega))$-solution to the transposition system 
\bel{eq3}
\left\{\begin{array}{ll}
(i\pd_t v +\Delta_A +q )v=F, & \mbox{in}\ Q,\\  
v(T,\cdot)=0, & \mbox{in}\ \Omega,\\ 
v=0, & \mbox{on}\ \Sigma.
\end{array}
\right.
\ee
We refer to Subsection \ref{sec-ts} for the full definition and description of transposition solutions to \eqref{1.1}.

Since $\pd \Omega$ is not bounded, we introduce the following notations. First, we set
$$
H^s(\pd \Omega) := H_{x_3}^s(\R,L^2(\pd \omega)) \cap L_{x_3}^2(\R,H^s(\pd \omega)),\ s>0,
$$
where $x_3$ denotes the longitudinal variable of $\Omega$. 
Next, we put
$$
H^{r,s}((0,T)\times X) :=H^r(0,T;L^2(X))\cap L^2(0,T;H^s(X)),\ r ,s >0, 
$$
where $X$ is either $\Omega$ or $\pd \Omega$. For the sake of shortness, we write $H^{r,s}(Q)$ (resp., $H^{r,s}(\Sigma)$) instead of $H^{r,s}((0,T) \times \Omega)$ (resp., $H^{r,s}((0,T) \times \pd \Omega)$). Finally, we define
$$H^{2,1}_0(\Sigma) := \{ f\in H^{2,1}(\Sigma);\ f(0,\cdot)=\pd_t f(0,\cdot)=0 \}, $$
and state the existence and uniqueness result of solutions to \eqref{1.1} in the transposition sense, as follows.

\begin{theorem}
\label{Th0} 
For $M>0$, let $A\in W^{1,\infty}(\Omega,\R)^3$ and $q\in W^{1,\infty}(\Omega,\R)$ satisfy the condition
\bel{c-bded}
\norm{A}_{W^{1,\infty}(\Omega)^3}+\norm{q}_{W^{1,\infty}(\Omega)} \leq M.
\ee
Then, for each $f\in H_0^{2,1}(\Sigma)$, the IBVP \eqref{1.1} admits a unique solution in the transposition sense $u \in H^1(0,T;H^1(\Omega))$, and the estimate
\bel{t1a}
\norm{u}_{H^1(0,T;H^1(\Omega))} \leq C \norm{f}_{H^{2,1}(\Sigma)},
\ee
holds for some positive constant $C$ depending only on $T$, $\omega$ and $M$.
Moreover, the normal derivative $\pd_\nu u \in L^2(\Sigma)$, and we have
\bel{t1b}
\norm{\pd_\nu u}_{L^2(\Sigma)}\leq C \norm{f}_{H^{2,1}(\Sigma)}.
\ee
\end{theorem}

It is clear from the definition \eqref{DN} and the continuity property \eqref{t1b}, that the DN map $\Lambda_{A,q}$ belongs to $\cB(H^{2,1}_0(\Sigma), L^2(\Sigma))$, the set of linear bounded operators from $H^{2,1}_0(\Sigma)$ into $L^2(\Sigma)$.

\subsection{Non uniqueness}
There is a natural obstruction to the identification of $A$ by $\Lambda_{A,q}$, arising from the invariance of the DN map under gauge transformation. More precisely, if $\Psi \in W^{2,\infty}(\Omega)$ verifies $\Psi|_{\Gamma}=0$, then we have $u_{A+\nabla \Psi}=e^{-i \Psi}u_A$, where $u_A$ (resp., $u_{A+\nabla \Psi}$) denotes the solution to \eqref{1.1} associated with the magnetic potential $A$ (resp., $A+\nabla \Psi$), $q\in L^\infty(\Omega)$ and $f\in H^{2,1}_0(\Sigma)$. Further, as
$$ (\pd_\nu +i (A + \nabla \Psi) \cdot\nu) u_{A + \nabla \Psi} = e^{-i \Psi} (\pd_\nu +i A \cdot \nu) u_A=(\pd_\nu +i A\cdot\nu) u_A\ \mbox{on}\ \Sigma, $$
by direct calculation, we get that $\Lambda_{A,q}=\Lambda_{A + \nabla \Psi,q}$, despite of the fact that the two potentials $A$ and $A+\nabla \Psi$ do not coincide in $\Omega$ (unless $\psi$ is uniformly zero). 

This shows that the best we can expect from the knowledge of the DN map is to identify $(A,q)$ modulo gauge transformation of $A$. When $A_{|\pd \Omega}$ is known, this may be equivalently reformulated as to whether the magnetic field defined by the 2-form
$$
\d A := \sum_{i,j=1}^3 \para{\pd_{x_j} a_i - \pd_{x_i} a_j} \d x_j \wedge \d x_i,
$$
and the electric potential $q$, can be retrieved by $\Lambda_{A,q}$. This is the inverse problem that we examine in the remaining part of this text.

\subsection{Main results}

We define the set of admissible magnetic potentials as
$$
\cA := \left\{ A=(a_i)_{1 \leq i \leq 3};\ a_1,a_2\in L_{x_3}^\infty(\R, H^2_0(\omega))\cap W^{2,\infty}(\Omega)\ \mbox{and}\ a_3 \in C^3(\overline{\Omega})\ \mbox{satisfies}\ \eqref{decay}-\eqref{bord} \right\}, 
$$
where
\bel{decay}
\sup_{x \in \Omega} \left( \sum_{\alpha \in \mathbb N^3,|\alpha|\leq 3} \langle x_3 \rangle^d |\pd_x^\alpha a_3(x)| \right) <\infty\ \mbox{for some}\ d>1,
\ee
and 
\bel{bord}
\pd^\alpha_x a_3(x)=0,\ x \in \pd \Omega,\ \alpha \in \mathbb N^3\ \mbox{such that}\ |\alpha| \leq2.
\ee
Here $H^2_0(\omega)$ denotes the closure of $\mathcal C^\infty_0(\omega)$ in the $H^2(\omega)$-topology, and $\langle x_3 \rangle :=(1+x_3^2)^{1 \slash 2}$.

The first result of this paper claims stable determination of the magnetic field $\d A$ and unique identification of electric potential $q$, from the knowledge of the full data, i.e. the DN map defined by \eqref{DN}, where both the Dirichlet and Neumann measurements are performed on the whole boundary $\Sigma$.

\begin{theorem}
\label{Th1}
Fix $A_*:=(a_{i,*})_{1 \leq i \leq 3} \in W^{2,\infty}(\Omega,\R)^3$, and for $j=1,2$, let $q_j \in W^{1,\infty}(\Omega)$, and
$A_j:=(a_{i,j})_{1 \leq i \leq 3} \in A_* + \mathcal{A}$, satisfy the condition:
\bel{t3a}
\sum_{i=1}^2 \pd_{x_i} \left( \pd_{x_3} (a_{i,1}-a_{i,2})-\pd_{x_i}(a_{3,1}-a_{3,2})\right) = 0,\ \mbox{in}\ \Omega.
\ee
Then, $\Lambda_{A_1,q_1}=\Lambda_{A_2,q_2}$ yields $(\d A_1,q_1)=(\d A_2,q_2)$.\\
Assume moreover that the estimate
\bel{t3b}
\sum_{j=1}^2 \left(\norm{A_j}_{W^{2,\infty}(\Omega)}+ \norm{q_j}_{W^{1,\infty}(\Omega)}+\norm{e_j}_{W^{3,\infty}(\Omega)}\right)+\norm{A_*}_{W^{2,\infty}(\Omega)} \leq M, 
\ee
holds for some $M>0$, with 
$$e_j(x',x_3):=\int_{-\infty}^{x_3} (a_{3,j}(x',y_3)-a_{3,*}(x',y_3)) \d y_3,\ (x',x_3) \in \Omega.$$ 
Then there exist two constants $\mu_0\in (0,1)$ and $C>0$, both of them depending only on $T$, $\omega$ and $M$, such that we have
\bel{t3bb}
\norm{\d A_1- \d A_2}_{L_{x_3}^\infty(\R,L^2(\omega))}\leq C \norm{\Lambda_{A_1,q_1}-\Lambda_{A_2,q_2}}^{\mu_0}.
\ee
\end{theorem}
In \eqref{t3bb} and in the remaining part of this text, $\| \cdot \|$ denotes the usual norm in $\cB(H^{2,1}(\Sigma), L^2(\Sigma))$.
Notice that in Theorem \ref{Th1} we make use of the full DN map, as the magnetic field $\d A$ and the electric potential $q$ are recovered by observing the solution to \eqref{1.1} on the entire lateral boundary $\Sigma$. In this case we may consider general unknown coefficients, in the sense that the behavior of $A$ and $q$ with respect to the infinite variable is not prescribed (we only assume that these coefficients and their derivatives are uniformly bounded in $\Omega$). In order to achieve the same result by measuring on a bounded subset of $\Sigma$ only, we need some extra information on the behavior of the unknown coefficients with respect to $x_3$. Namely, we impose that the strength of the magnetic field generated by $A=(a_i)_{1 \leq i \leq 3}$, reaches its maximum in the bounded subset $(-r,r) \times \omega$ of $\Omega$, for some fixed $r>0$, i.e.
\bel{BB}
\| \pd_{x_i} a_j - \pd_{x_j} a_i \|_{L_{x_3}^\infty(\R,L^2(\omega))} = \| \pd_{x_i} a_j - \pd_{x_j} a_i \|_{L_{x_3}^\infty(-r,r;L^2(\omega))},\ i, j=1,2,3.
\ee
Thus, with reference to \eqref{BB}, we set $\Gamma_r := \pd \omega \times (-r,r)$, introduce the space
$$ H^{2,1}_0((0,T) \times \Gamma_r) :=\{ f \in H^{2,1}(\Sigma); f(0,\cdot)=\pd_t f(0,\cdot)=0\ \mbox{and}\ \supp\ f\subset [0,T] \times \pd \omega \times [-r,r] \}, $$
and define the partial DN map $\Lambda_{A,q,r}$, by
$$ \Lambda_{A,q,r}(f) := (\pd_\nu +i A \cdot \nu) u_{|(0,T) \times \Gamma_r},\  f \in H^{2,1}_0((0,T) \times \Gamma_r), $$
where $u$ denotes the solution to \eqref{1.1}. The following result states for each $r>0$, that the magnetic field induced by potentials belonging (up to an additive $W^{2,\infty}(\Omega,\R)^3$-term) to 
$$\cA_r:=\{ A=(a_i)_{1 \leq i \leq 3} \in \cA\ \mbox{satisfying}\ \eqref{BB} \}, $$ 
can be retrieved from the knowledge of the partial DN map $\Lambda_{A,q,r'}$, provided we have $r'>r$.

\begin{theorem}
\label{Th2}
For $j=1,2$, let $q_j \in W^{1,\infty}(\Omega,\R)$, and let $A_j \in W^{2,\infty}(\Omega,\R)^3$ satisfy $A_1-A_2 \in \cA_r$, for some $r>0$. Suppose that there exists $r'>r$, such that
$\Lambda_{A_1,q_1,r'}=\Lambda_{A_2,q_2,r'}$. Then, we have $\d A_1 = \d A_2$. Furthermore, if
$$
\norm{q_1-q_2}_{L_{x_3}^\infty(\R, H^{-1}(\omega))}=\norm{q_1-q_2}_{L_{x_3}^\infty(-r,r; H^{-1}(\omega))},
$$
we have in addition $q_1 = q_2$.  \\
Assume moreover that \eqref{t3a}-\eqref{t3b} hold. Then, the estimate
\bel{n-t3bb}
\norm{\d A_1 - \d A_2}_{L_{x_3}^\infty(\R, L^2(\omega))^3}\leq C \norm{\Lambda_{A_1,q_1, r'}-\Lambda_{A_2,q_2,r'}}^{\mu_1},
\ee
holds with two constants $C>0$, and $\mu_1 \in (0,1)$, depending only on $T$, $\omega$, $M$, $r$ and $r'$.
\end{theorem}

We stress out that Theorem \ref{Th2} applies not only to magnetic (resp., electric) potentials $A_j$ (resp., $q_j$), $j=1,2$, which coincide outside $\omega \times (-r,r)$, but to a fairly more general class of magnetic potentials, containing, e.g., $2r$-periodic potentials with respect to $x_3$. More generally, if $g \in W^{2,\infty}(\R,\R_+)$ (resp. $g \in W^{1,\infty}(\R,\R_+)$) is an even and non-increasing function in $\R_+$, then it is easy to see that potentials of the form $g \times A_j$ (resp., $g \times q_j$), where $A_j$ (resp., $q_j$) are suitable $2r$-periodic magnetic (resp., electric) potentials with respect to $x_3$, fulfill the conditions of Theorem \ref{Th2}.

Notice that the absence of stability for the electric potential $q$, manifested in both Theorems \ref{Th1} and \ref{Th2}, arises from the infinite extension of the spatial domain $\Omega$ in the $x_3$ direction. Indeed, the usual derivation of a stability equality for $q$, from estimates such as \eqref{t3bb} or \eqref{n-t3bb}, requires that the differential operator $\d$ be invertible in $\Omega$. Such a property is true in bounded domains (see e.g. \cite{[Tzou]}), but, to the best of our knowledge, it is not known whether it can be extended to unbounded waveguides. One way to overcome this technical difficulty is to impose certain gauge condition on the magnetic potentials, by prescribing their divergence.
In this case, we establish in Theorem \ref{Th22}, below, that the electric and magnetic potentials can be simultaneously and stably determined by the DN map.

\subsubsection{Simultaneous stable recovery of magnetic and electric potentials}  
We first introduce the set of divergence free transverse magnetic potentials,
$$
\cA_0 := \{ A = (a_1,a_2,0);\ a_1,a_2 \in L_{x_3}^\infty(\R, H^2_0(\omega))\cap W^{2,\infty}(\Omega),\ \pd_{x_1} a_1+ \pd_{x_2} a_2=0\ \mbox{in}\ \Omega \},
$$
in such a way that we have $\nabla \cdot A = \nabla \cdot A_*$ for any $A \in A_*+\cA_0$. Here $A^* \in W^{2,\infty}(\Omega)^3$ is an arbitrary fixed magnetic potential.
Since identifying $A \in A_*+\cA_0$ from the knowledge of the DN map, amounts to determining the magnetic field $\d A$, we have the following result.

\begin{theorem}
\label{Th22} 
Let $M>0$, and let $A_* \in W^{2,\infty}(\Omega,\R)^3$. For $j=1,2$, let $q_j \in W^{1,\infty}(\Omega,\R)$, and let $A_j \in A_*+\cA_0$ satisfy \eqref{t3b}. Then, there exist two constant $\mu_2 \in (0,1)$ and $C=C(T,\omega,M)>0$, such that we have
\bel{t22a}
\norm{A_1-A_2}_{L_{x_3}^\infty(\R,L^2(\omega))^3} + \norm{q_1-q_2}_{L_{x_3}^\infty(\R,H^{-1}(\omega))} \leq C\norm{\Lambda_{A_1,q_1}-\Lambda_{A_2,q_2}}^{\mu_2}.
\ee
Assume moreover that the two following conditions
\bel{Ma}
\norm{A_1-A_2}_{L_{x_3}^\infty(\R,L^2(\omega))^3}=\norm{A_1-A_2}_{L_{x_3}^\infty(-r,r;L^2(\omega))^3},
\ee
and 
\bel{el}
\norm{q_1-q_2}_{L_{x_3}^\infty(\R,H^{-1}(\omega))}=\norm{q_1-q_2}_{L_{x_3}^\infty(-r,r;H^{-1}(\omega))},
\ee
hold simultaneously for some $r>0$. Then, for each $r'>r$, we have
\bel{t22c}
\norm{A_1-A_2}_{L_{x_3}^\infty(\R,L^2(\omega))}+\norm{q_1-q_2}_{L_{x_3}^\infty(\R,H^{-1}(\omega))} \leq C \norm{\Lambda_{A_1,q_1, r'}-\Lambda_{A_2,q_2,r'}}^{\mu_2},
\ee
where $C$ is a positive constant depending only on $T$, $\omega$, $M$, $r$ and $r'$.
\end{theorem}

\subsubsection{Comments}
The key ingredient in the analysis of the inverse problem under examination is a suitable set of GO solutions to the magnetic Schr\"odinger equation appearing in \eqref{1.1}. These functions are specifically designed for the waveguide geometry of $\Omega$, in such a way that the unknown coefficients can be recovered by a separation of variables argument.
More precisely, we seek GO solutions that are functions of $x=(x',x_3) \in \Omega$, but where the transverse variable $x' \in \omega$ and the translational variable $x_3 \in \R$ are separated. This approach was already used in \cite{[Ki]}, for determining zero order unknown coefficients of the wave equation. Since we consider first order unknown coefficients in this paper, the main issue here is to take into account both the cylindrical shape of $\Omega$ and the presence of the magnetic potential, in the design of the GO solutions.

When the domain $\Omega$ is bounded, we know from \cite{[BC2]} that the magnetic field $\d A$ is uniquely determined by the DN map associated with \eqref{1.1}. The main achievement of the present paper is to extend the above statement to unbounded cylindrical domains. Actually, we also improve the results of \cite{[BC2]} in two directions. First, we prove simultaneous determination of the magnetic field $\d A$ and the electric potential $q$. Second, the regularity condition imposed on admissible magnetic potentials entering the Schr\"odinger equation of \eqref{1.1}, is weakened from $W^{3,\infty}(\Omega)$ to $W^{2,\infty}(\Omega)$.

To our best knowledge, this is the first mathematical paper claiming identification by boundary measurements, of non-compactly supported magnetic field and electric potential. Moreover, in contrast to the other works \cite{[BKS],[CKS1],[KPS1]} dealing with the stability issue of inverse problems for the Schr\"odinger equation in an infinite cylindrical domain, available in the mathematics literature, here we no longer require that the various unknown coefficients be periodic, or decay exponentially fast, in the translational direction of the waveguide. 

Finally, since the conditions \eqref{BB} and \eqref{Ma}-\eqref{el} are imposed in $\omega \times (-r,r)$ only, and since the solution to \eqref{1.1} lives in the infinitely extended cylinder $(0,T) \times \Omega$, we point out that the results of Theorems \ref{Th2} and \ref{Th22}
cannot be derived from similar statements derived in a bounded domain.

\subsection{Outlines}
The paper is organized as follows. In Section 2 we examine the forward problem associated with \eqref{1.1}, by rigorously defining the transposition solutions to \eqref{1.1}, and proving Theorem \ref{Th0}. In Section 3, we build the GO solutions to the Schr\"odinger equation appearing in \eqref{1.1}, which are the key ingredient in the analysis of the inverse problem carried out in the two last sections of this paper. In Section 4, we estimate the X-ray transform of first-order partial derivatives of the transverse magnetic potential, and the Fourier transform of the aligned magnetic field, in terms of the DN map. Finally, Section 5 contains the proofs of Theorems \ref{Th1}, \ref{Th2}, and \ref{Th22}.

\section{Analysis of the forward problem}
\setcounter{equation}{0}

In this section we study the forward problem associated with \eqref{1.1}, that is, we prove the statement of Theorem \ref{Th0}. Although this problem is very well documented when $\Omega$ is bounded (see e.g. \cite{[BC2]}), to our best knowledge, it cannot be directly derived from any published mathematical work in the framework of the unbounded waveguide $\Omega$ under consideration in this paper.

The proof of Theorem \ref{Th0}, which is presented in Subsection \ref{sec-prT0}, deals with transposition solutions to \eqref{1.1}, that are rigorously defined in Subsection \ref{sec-ts}. As a preliminary, we start by examining in Subsection \ref{sec-elliptic}, the elliptic part of the dynamic magnetic Schr\"odinger operator appearing in \eqref{1.1}, and we establish an existence and uniqueness result for the corresponding system in Subsection \ref{sec-eur}.

\subsection{Elliptic magnetic Schr\"odinger operator}
\label{sec-elliptic}
For $A \in W^{1,\infty}(\Omega,\R)^3$, we set $\nabla_A:=\nabla+iA$, where $i A$ denotes the multiplier by $i A$, and notice for all $u \in H^1(\Omega)$, that
\bel{coercive}
| \nabla_A u(x) |^2 \geq (1-\eps) | \nabla u(x) | + (1-\eps^{-1}) | Au(x) |^2,\ \eps >0,\ \ x \in \Omega.
\ee
Next, for $q \in L^\infty(\Omega;\R)$, we introduce the sesquilinear form
$$
\mathbf{h}_{A,q}(u,v) :=\int_\Omega\nabla_A u(x)\cdot\overline{\nabla_A v}(x) \d x -\int_\Omega q(x) u(x) \overline{v}(x) \d x,\ u,v\in \D(\mathbf{h}_{A,q}):=H^1_0(\Omega),
$$
and consider the self-adjoint operator $\sH_{A,q}$ in $L^2(\Omega)$, generated by $\mathbf{h}_{A,q}$. In light of \cite[Proposition 2.5]{[KPS2]}, $\sH_{A,q}$ acts on its domain $\D(\sH_{A,q}):=H_0^1(\Omega) \cap H^2(\Omega)$, as the operator $-(\Delta_A+q)$, where $\Delta_A:= \nabla_A \cdot \nabla_A$ is expressed by \eqref{DelA}.

Further, for all $x \in \Omega$ fixed, taking $\epsilon=|A(x)|^2 \slash (1+|A(x)|^2)$ in \eqref{coercive}, we get that $| \nabla_A u(x) |^2 \geq | \nabla u(x) |^2 \slash (1 + |A(x)|^2) - |u(x)|^2$, whence
$$
\mathbf{h}_{A,0}(u,u) + \norm{u}_{L^2(\Omega)}^2 \geq \frac{\norm{\nabla u}^2_{L^2(\Omega)^3}}{1+\norm{A}_{L^\infty(\Omega)}^2},\ u \in H^1_0(\Omega),
$$
where $\mathbf{h}_{A,0}$ stands for $\mathbf{h}_{A,q}$ when $q$ is uniformly zero.
Thus, we deduce from the Poincar\'e inequality and Lax Milgram's theorem, that for any $v \in H^{-1}(\Omega)$, there exists a unique $\phi_v \in H^1_0(\Omega)$ satisfying
\bel{surjective}
-\Delta_A \phi_v+\phi_v=v.
\ee
Next, for $u$ and $v$ in $H^{-1}(\Omega)$, we put
$$
\langle u , v \rangle_{-1} :=\re\left(\int_\Omega \nabla_A \phi_u(x) \cdot \overline{\nabla_A \phi_v}(x) \d x + \int_\Omega \phi_u(x) \overline{\phi_v}(x) \d x \right),
$$
and check that the space $H^{-1}(\Omega)$, endowed with the above scalar product, is Hilbertian. Having said that, we may now prove the following technical result.

\begin{lemma}
\label{ll1} 
For each $A \in W^{1,\infty}(\Omega,\R)^3$, the linear operator $\sB_A:=\Delta_A$, with domain $\D(\sB_A):=H^1_0(\Omega)$, is self-adjoint and negative in $H^{-1}(\Omega)$.
\end{lemma}
\begin{proof} 
We proceed as in the proof of \cite[Proposition 2.6.14 and Corollary 2.6.15]{[CH]}. Namely, we pick $u$ and $v$ in $\cC^\infty_0(\Omega)$, and write
$$
\langle \sB_A u, v \rangle_{-1}=\langle w , v \rangle_{-1}+ \langle u , v\rangle_{-1},
$$
with $w:=\sB_A u-u$. Taking into account that $\phi_w=-u$, we obtain that
\bel{a0}
\langle \sB_A u,v \rangle_{-1}=-\re \left( \int_\Omega \nabla_A u(x) \cdot \overline{\nabla_A \phi_v}(x) \d x+ \int_\Omega u(x) \overline{\phi_v}(x) \d x \right)+ \langle u,v \rangle_{-1}.
\ee
Next, integrating by parts, we get that
$$
-\re \left( \int_\Omega \nabla_A u(x) \cdot \overline{\nabla_A \phi_v}(x) \d x + \int_\Omega u(x) \overline{\phi_v}(x) \d x\right)=-\re \langle u,-\Delta_A \phi_v+\phi_v \rangle_{L^2(\Omega)}=-\re \langle u,v \rangle_{L^2(\Omega)},
$$
so \eqref{a0} yields
\bel{2.22}
\langle \sB_A u , v \rangle_{-1}=-\re \langle u , v \rangle_{L^2(\Omega)} + \langle u , v \rangle_{-1}.
\ee
Further, since $\langle u , u \rangle_{-1}=\re \langle \phi_u,(-\Delta_A+1)\phi_u \rangle_{L^2(\Omega)}=\re \langle \phi_u, u \rangle_{L^2(\Omega)}$ 
and $\norm{\phi_u}_{L^2(\Omega)}^2 \leq \langle u , u \rangle_{-1}$,
we see that $\langle u , u \rangle_{-1} \leq \norm{u}_{L^2(\Omega)}^2$. Therefore, we obtain
\bel{dissipative}
\langle \sB_A u,u \rangle_{-1}=-\norm{u}_{L^2(\Omega)}^2+ \langle u , u \rangle_{-1} \leq 0,
\ee
by taking $v=u$ in \eqref{2.22}.

By density of $\mathcal C^\infty_0(\Omega)$ in $H^1_0(\Omega)$, both estimates \eqref{2.22} and \eqref{dissipative} remain valid for all $u$ and $v$ in $H^1_0(\Omega)$. As a consequence, the operator $\sB_A$ is dissipative. Furthermore, $1-\sB_A$ being surjective from $H^1_0(\Omega)$ onto $H^{-1}(\Omega)$, by \eqref{surjective}, we get that
$\sB_A$ is $m$-dissipative.
Moreover, as it follows readily from \eqref{2.22} that
$$
\langle \sB_A u , v \rangle_{-1}=\langle u , \sB_A v \rangle_{-1},\ u,v\in H^1_0(\Omega),
$$
we see that the graph of $\sB_A$ is contained into the one of its adjoint $\sB_A^*$. Therefore, $\sB_A$ is self-adjoint, in virtue of \cite[Corollary 2.4.10]{[CH]}.
\end{proof}

\subsection{Existence and uniqueness result}
\label{sec-eur}

For further use, we establish the following existence and uniqueness result for the system
\bel{eq2}
\left\{\begin{array}{ll}
(i\pd_t +\Delta_A + q ) v = F, & \mbox{in}\ Q,\\ 
 v(0,\cdot)=0, &\mbox{in}\ \Omega,\\ 
 v=0, &\mbox{on}\ \Sigma,
\end{array}
\right.
\ee
with homogeneous Dirichlet boundary condition and suitable source term $F$.

\begin{lemma}
\label{ll2} 
Let $M$, $A$ and $q$ be the same as in Theorem \ref{Th0}. 
\begin{enumerate}[(i)]
\item Assume that $F \in L^1(0,T;H_0^1(\Omega))$. Then, the system
\eqref{eq2} admits a unique solution $v \in \cC([0,T],H_0^1(\Omega))$, satisfying
\bel{ll2z}
\| v \|_{\cC([0,T],H^1(\Omega)} \leq C \| F \|_{L^1(0,T;H^1(\Omega))},
\ee
for some constant $C>0$, depending only on $T$, $\omega$ and $M$.
\item If $F \in W^{1,1}(0,T;L^2(\Omega))$, then \eqref{eq2} admits a unique solution
$$ v \in \cZ:= \cC^1([0,T],L^2(\Omega)) \cap \cC([0,T],H_0^1(\Omega) \cap H^2(\Omega)), $$
and there exists $C=C(T,\omega,M)>0$, such that
$$
\norm{v}_{\cZ} \leq C \norm{F}_{W^{1,1}(0,T;L^2(\Omega))}.
$$
\end{enumerate}
\end{lemma}

\begin{proof}
The proof boils down on the following statement, borrowed from \cite[Lemma 2.1]{[CKS1]}.

Let $X$ be a Banach space, $U$ be a m-dissipative operator in $X$ with dense domain $D(U)$
and $B \in \cC([0,T],\cB(D(U)))$. Then for all $v_0 \in D(U)$ and
$f\in \cC([0,T],X)\cap L^1(0,T;D(U))$ $($resp. $f\in W^{1,1}(0,T;X))$ there is a unique solution
$v \in Z_0=\cC([0,T], D(U))\cap C^1([0,T],X)$ to the following Cauchy problem
$$
\left\{
\begin{array}{ll}
v'(t)=U v(t)+ B(t)v(t)+f(t),
\\
v(0)=v_0,
\end{array}
\right.
$$
such that
$$
\|v\|_{Z_0} =\|v\|_{C^0([0,T],D(U))}+\|v\|_{C^1([0,T],X)} \leq C(\|v_0\|_{D(U)}+\|f\|_\ast).
$$
Here $C$ is some positive constant depending only on $T$ and $\|B\|_{C ([0,T],\cB(D(U)))}$, and 
$\|f\|_\ast$ stands for the norm $\| f \|_{C([0,T], X)\cap L^1(0,T;D(U))}$ $($resp. $\|f\|_{W^{1,1}(0,T;X)})$.

Notice that the operator $i \sB_{A}$ is skew-adjoint, since $\sB_{A}$ is self-adjoint in $H^{-1}(\Omega)$. Hence $i \sB_{A}$ is $m$-dissipative with dense domain in $H^{-1}(\Omega)$. Further, the multiplier by $i q$ being bounded in $\cC[0,T],H_0^1(\Omega)$, we obtain (i) 
by applying the above result with $X=H^{-1}(\Omega)$, $U=i \sB_{A,q}$, $f=i F$, $B(t)=i q$, and $v_0=0$.

Similarly, as $\sH_{A,q}$ is self-adjoint in $L^2(\Omega)$, then the operator $-i \sH_{A,q}$ is $m$-dissipative with dense domain in $L^2(\Omega)$, so we derive (ii) by applying
\cite[Lemma 2.1]{[CKS1]} with $X=L^2(\Omega)$, $U=-i \sH_{A,q}$, $f=i F$, $B(t)=0$, and $v_0=0$. 
\end{proof}

\begin{remark}
\label{rmk-eur}
Since $w(t,x):=\overline{v(T-t,x)}$, for $(t,x) \in Q$, is solution to
\bel{eq2t}
\left\{\begin{array}{ll}
(i\pd_t +\Delta_A + q ) w = F, & \mbox{in}\ Q,\\ 
w(T,\cdot)=0, &\mbox{in}\ \Omega,\\ 
w=0 &\mbox{on}\ \Sigma,
\end{array}
\right.
\ee
whenever $v$ is solution to the IBVP \eqref{eq2}, where the function $(t,x) \mapsto \overline{F(T-t,x)}$ is substituted for $F$, we infer from Lemma \ref{ll2} that the transposed system \eqref{eq2t} admits a unique solution $w$ in $\cC^0([0,T],H_0^1(\Omega))$ (resp., $\cZ$) provided $F$ is in $L^1(0,T;H_0^1(\Omega))$ (resp., $W^{1,1}(0,T;L^2(\Omega))$).
\end{remark}

\subsection{Transposition solutions}
\label{sec-ts}
As a preamble to the definition of transposition solutions to \eqref{1.1}, we establish that the normal derivative of the $\cC([0,T],H_0^1(\Omega))$-solution to \eqref{eq2} is lying in $L^2(\Sigma)$.

\begin{lemma}
\label{ll4} 
Let $M$, $A$ and $q$ be as in Lemma \ref{ll2}. 
Then, the linear map $F \mapsto \pd_\nu v$, where $v$ denotes the $\cC([0,T],H_0^1(\Omega))$-solution to \eqref{eq2} associated with $F \in L^1(0,T;H_0^1(\Omega))$, given by Lemma \ref{ll2}, is bounded from $L^1(0,T;H_0^1(\Omega))$ into $L^2(\Sigma)$.
\end{lemma}
\begin{proof} 
Since $\norm{v}_{\cC([0,T],H^1(\Omega))} \leq C \norm{F}_{L^1(0,T;H^1(\Omega))}$, by \eqref{ll2z}, we may assume without loss of generality that $A=0$ and $q=0$.

Assume that $F \in W^{1,1}(0,T;L^2(\Omega))$ in such a way $v \in \cZ$, in virtue of Lemma \ref{ll2}.
Let $N_1 \in \cC^2(\overline{\omega})^2$ satisfy $N_1=\nu_1$ on $\pd \omega$, where $\nu_1$ denotes the unit outward
normal vector to $\pd \omega$. Put $N(x',x_3):=(N_1(x'),0)$ for all $x'\in\omega$ and  $x_3\in\R$, so that $N \in \cC^2(\overline{\Omega})^3\cap W^{2,\infty}(\Omega)^3$ verifies
$N=\nu$ on $\pd \Omega$. Then, we have
\bel{ll4a}
\langle i \pd_t v+\Delta  v , N \cdot \nabla v \rangle_{L^2(Q)} = \langle F , N \cdot \nabla v \rangle_{L^2(Q)}.
\end{equation}
By integrating by parts with respect to $t$, we get 
\bea
\langle \pd_t v , N \cdot \nabla v \rangle_{L^2(Q)} & = & \langle v(T,\cdot) , N \cdot \nabla v(T,\cdot) \rangle_{L^2(\Omega)} - \langle v , N \cdot \nabla \pd_t v \rangle_{L^2(Q)}\nonumber\\
& = & \langle v(T,\cdot) , N \cdot \nabla v(T,\cdot) \rangle_{L^2(\Omega)} + \langle  N \cdot \nabla v , \pd_t v\rangle_{L^2(Q)} - I,
\label{a1}
\eea
where $I := \int_Q N \cdot \nabla (v \overline{\pd_t v}) \d x \d t$. Taking into account that  $N \cdot \nabla = N_1 \cdot \nabla_{x'}$, where $\nabla_{x'}$ denotes the gradient operator with respect to $x' \in \omega$, we have $I=\int_Q N_1 \cdot \nabla_{x'} (v \overline{\pd_t v}) \d x \d t$, hence
\bea
I & = &  \int_{Q} \nabla_{x'} \cdot (v(t,x) \overline{\pd_t v}(t,x) N_1(x')) \d x' \d x_3 \d t - \langle (\nabla \cdot N) v , \overline{v} \pd_t v \rangle_{L^2(Q)} \nonumber \\
& = & \int_{\Sigma} v(t,x) \overline{\pd_t v}(t,x) N_1(x') \cdot \nu_1(x') \d x' \d x_3 \d t - \langle (\nabla \cdot N) v , \pd_t v \rangle_{L^2(Q)} \nonumber \\
& = & - \langle (\nabla \cdot N) v , \pd_t v \rangle_{L^2(Q)}, \label{a2}
\eea
by Green's formula, since $v_{\vert \Sigma}=0$. Putting \eqref{a1}-\eqref{a2} together, we obtain that
\bea
& & 2 \re \langle i \pd_t v , N \cdot \nabla v \rangle_{L^2(Q)} \nonumber \\
& = & i \langle  v(T,\cdot) , N \cdot \nabla v(T,\cdot) \rangle_{L^2(\Omega)}  - \langle (\nabla \cdot N) v , i \pd_t v \rangle_{L^2(Q)} \nonumber \\
& = & i \langle  v(T,\cdot) , N \cdot \nabla v(T,\cdot) \rangle_{L^2(\Omega)} + \langle (\nabla \cdot N) v , \Delta v \rangle_{L^2(Q)} - \langle (\nabla \cdot N) v , F \rangle_{L^2(Q)}. \label{ll4b}
\eea
Applying the Green formula with respect to $x' \in \omega$ and integrating by parts with respect to $x_3 \in \R$, we find that
$$
\langle (\nabla \cdot N) v , \Delta v \rangle_{L^2(Q)} = 
- \langle (\nabla \cdot N) \nabla v , \nabla v \rangle_{L^2(Q)^3} - \langle v \nabla (\nabla \cdot N), \nabla v \rangle_{L^2(Q)^3},
$$
so \eqref{ll4b} entails
\beas
2 \re \langle i \pd_t v , N \cdot \nabla v \rangle_{L^2(Q)} 
&=& i \langle  v(T,\cdot) , N \cdot \nabla v(T,\cdot) \rangle_{L^2(\Omega)}  - \langle (\nabla \cdot N) \nabla v , \nabla  v \rangle_{L^2(Q)^3} \\
& & - \langle v \nabla (\nabla \cdot N), \nabla v \rangle_{L^2(Q)^3} - \langle (\nabla \cdot N) v , F \rangle_{L^2(Q)}.
\eeas
This and \eqref{ll2z} yield
\beas
\abs{\re \langle i \pd_t v , N \cdot \nabla v \rangle_{L^2(Q)}} & \leq &  C \norm{v}_{\cC([0,T],H^1(\Omega))} \left( \norm{v}_{\cC([0,T],H^1(\Omega))} + \norm{F}_{L^1(0,T;H^1(\Omega))} \right) \\
& \leq & C\norm{F}_{L^1(0,T;H^1(\Omega))}^2.
\eeas
From this and \eqref{ll4a}, it then follows that
\bel{ll4c}
\abs{\re \langle \Delta  v , N \cdot \nabla v \rangle_{L^2(Q)}} \leq  C \norm{F}_{L^1(0,T;H^1(\Omega))}^2
\ee
On the other hand, we get upon applying the Green formula with respect to $x' \in \omega$ and integrating by parts with respect to $x_3 \in \R$, that
\bea
\langle \Delta  v , N \cdot \nabla v \rangle_{L^2(Q)} & = & - \langle \nabla v , \nabla (N \cdot \nabla v) \rangle_{L^2(Q)^3}
+ \langle \nabla v \cdot \nu , N \cdot \nabla v \rangle_{L^2(\Sigma)} \nonumber \\
& = & - \langle \nabla v , \nabla (N \cdot \nabla v) \rangle_{L^2(Q)^3} + \| \pd_\nu v \|_{L^2(\Sigma)}^2. \label{a3}
\eea
Moreover, since
$\re \left( \nabla v \cdot \nabla (N \cdot \overline{\nabla v}) \right) = \re \left( (H \nabla v) \cdot \overline{\nabla v} \right) +\frac{1}{2} N \cdot \nabla \abs{\nabla v}^2$ with $H:=(\pd_{x _i} N_j)_{1 \leq i , j \leq 3}$ and $N:=(N_j)_{1 \leq j \leq 3}$,
we infer from \eqref{a3} that
\bel{a4}
\re \langle \Delta  v , N \cdot \nabla v \rangle_{L^2(Q)} 
= \| \pd_\nu v \|_{L^2(\Sigma)}^2 - \re \langle H \nabla v , \nabla v \rangle_{L^2(Q)^3}
-\frac{1}{2} \int_Q N \cdot \nabla \abs{\nabla v}^2 \d x \d t. 
\ee
Further, by applying once more the Green formula with respect to $x' \in \omega$, we find for a.e. $(t,x_3) \in (0,T) \times \R$, that
\bea
& & \int_\omega N(x',x_3) \cdot \nabla \abs{\nabla v(t,x',x_3)}^2 \d x'= \int_\omega N_1(x') \cdot \nabla_{x'} \abs{\nabla v(t,x',x_3)}^2 \d x' \nonumber \\
& = & \| \nabla v(t,\cdot,x_3) \|_{L^2(\pd \omega)^3}^2 - \langle (\nabla \cdot N) \nabla v(t,\cdot,x_3) , \nabla v(t,\cdot,x_3) \rangle_{L^2(\omega)^3}. \label{a5}
\eea
Bearing in mind that $v_{\vert \Sigma}=0$, we have $\abs{\nabla v}^2=\abs{\pd_\nu v}^2$ on $\Sigma$, so we deduce from \eqref{a5} that
$$
\int_Q N \cdot \nabla \abs{\nabla v}^2 \d x \d t = \| \pd_\nu v \|_{L^2(\Sigma)}^2 - \langle (\nabla \cdot N) \nabla v , \nabla v \rangle_{L^2(Q)^3}.
$$
From this and \eqref{a4}, it then follows that
$$
\| \pd_\nu v \|_{L^2(\Sigma)}^2 = 2 \re \langle  \Delta v ,   N \cdot \nabla v \rangle_{L^2(Q)} + 2 \re \langle H \nabla v , \nabla v \rangle_{L^2(Q)^3}
- \langle (\nabla \cdot N) \nabla v , \nabla v \rangle_{L^2(Q)^3},
$$
and hence
$$
\norm{\pd_\nu v}_{L^2(\Sigma))}\leq C \left( \norm{F}_{L^1(0,T;H^1(\Omega))}+\norm{v}_{\mathcal{C}([0,T],H^1(\Omega))} \right) \leq C \norm{F}_{L^1(0,T;H^1(\Omega))},
$$
according to \eqref{ll2z} and \eqref{ll4c}.
By density of $W^{1,1}(0,T;H_0^1(\Omega))$ in $L^1(0,T;H^1_0(\Omega))$, it is clear that the above estimate extends to every $F \in L^1(0,T;H^1_0(\Omega))$, which proves the desired result.
\end{proof}

Armed with Lemma \ref{ll4}, we now introduce the transposition solution to \eqref{1.1}.
For $F \in L^1(0,T;H_0^1(\Omega))$, we denote by $v \in \cC^0([0,T],H_0^1(\Omega))$ the solution to \eqref{eq2t}, given by Remark \ref{rmk-eur}. Since $(t,x) \mapsto \overline{v(T-t,x)}$ is solution to \eqref{eq2} associated with the source term $(t,x) \mapsto \overline{F(T-t,x)}$, we infer from Lemma \ref{ll4}, that the mapping $F \mapsto \pd_\nu v$ is bounded from $L^1(0,T;H_0^1(\Omega))$ into $L^2(\Sigma)$. Therefore, for each $f \in L^2(\Sigma)$, the mapping
$$
\ell_f : F \mapsto \langle f , \pd_\nu v \rangle_{L^2(\Sigma)}, 
$$
is an anti-linear form on $L^1(0,T;H_0^1(\Omega))$. 
Thus, there exists a unique $u \in L^\infty(0,T;H^{-1}(\Omega))$ such that we have
\bel{tran}
\langle u , F \rangle_{L^\infty(0,T;H^{-1}(\Omega)),L^1(0,T;H^1_0(\Omega))}=\ell_f(F),\ F \in L^1(0,T;H_0^1(\Omega)),
\ee
according to Riesz's representation theorem.
The function $u$, characterized by \eqref{tran}, is named the solution in the transposition sense to \eqref{1.1}.

\subsection{Proof of Theorem \ref{Th0}} 
\label{sec-prT0}
Let $w \in L^\infty(0,T;H^{-1}(\Omega))$ be the solution in the transposition sense to the system
$$
\left\{
\begin{array}{ll}
\para{i\pd_t+\Delta_A+q} w=0,  & \mbox{in}\ Q,\\
w(0,\cdot)=0, & \mbox{in}\ \Omega,\\
w= \pd_t^2 f, & \mbox{on}\ \Sigma.
\end{array}
\right.
$$
For any $t \in (0,T)$, put $v(t,\cdot):= \int_0^t w(s,\cdot) \d s$, in such a way that $v$ is the solution in the transposition sense to the system
\bel{a6}
\left\{
\begin{array}{ll}
\para{i\pd_t+\Delta_A+q} v=0,  & \mbox{in}\ Q,\\
v(0,\cdot)=0, & \mbox{in}\ \Omega,\\
v= \pd_t f, & \mbox{on}\ \Sigma.
\end{array}
\right.
\ee
We have $v=\pd_t f \in H^{1,1/2}(\Sigma)$ by \cite[Section 4, Proposition 2.3]{[LM2]}, and since $H^{1,1/2}(\Sigma) \subset L^2(0,T; H^{1/2}(\pd \Omega))$,
and $-\Delta_A v=  iw + qv$ in $Q$, from the first line of \eqref{a6}, then $v\in L^2(0,T;H^1(\Omega)) \cap W^{1,\infty}(0,T;H^{-1}(\Omega))$. Moreover, we have the following estimate
\bel{a7}
\norm{v}_{L^2(0,T;H^1(\Omega))} \leq C \left( \norm{w}_{L^2(0,T;H^{-1}(\Omega))} + \norm{qv}_{L^2(0,T;H^{-1}(\Omega))} + \norm{\pd_t f}_{L^2(0,T; H^{1/2}(\pd \Omega))} \right),
\ee
where the constant $C>0$ depends only on $T$, $\omega$, and $M$. 

On the other hand, from the very definition of the transposition solution $w$, we obtain
\bel{a8}
\norm{w}_{L^2(0,T;H^{-1}(\Omega))}\leq T^{1/2}\norm{w}_{L^\infty(0,T;H^{-1}(\Omega))}\leq C \norm{\partial_t^2f}_{L^2(\Sigma)} \leq C\norm{f}_{H^{2,1}(\Sigma)}, 
\ee
with the aid of Lemma \ref{ll4}. As a consequence we have
\bel{a9}
\norm{q v}_{L^2(0,T;H^{-1}(\Omega))}\leq \norm{q}_{W^{1,\infty}(\Omega)} T \norm{w}_{L^2(0,T;H^{-1}(\Omega))} \leq C \norm{f}_{H^{2,1}(\Sigma)}.
\ee
Putting \eqref{a7}--\eqref{a9} together, we find that
\bel{a10}
\norm{v}_{L^2(0,T;H^1(\Omega))} \leq C\norm{f}_{H^{2,1}(\Sigma)},
\ee
for some constant $C=C(T,\omega,M)>0$. 

Finally, as $u(t)=\int_0^t v(s) ds$ is solution to \eqref{1.1} in the transposition sense, we have
$$ \norm{u}_{H^1(0,T; H^1(\Omega))} \leq (1+T)^{1 \slash 2} \norm{v}_{L^2(0,T;H^1(\Omega))}, $$
hence \eqref{t1a} follows from this and \eqref{a10}.

We turn now to proving \eqref{t1b}. To do that, we pick $f \in \cC^\infty([0,T] \times \pd \Omega) \cap H_0^{2,1}(\Sigma)$, and proceed as in the derivation of Lemma \ref{ll4}. We get that
$$
\norm{\pd_\nu u}_{L^2(\Sigma)} \leq C \left( \norm{u}_{H^1(0,T;H^1(\Omega))} + \norm{f}_{H^{2,1}(\Sigma)} \right),
$$
for some constant $C=C(T,\omega,M)>0$, so we deduce from \eqref{t1a} that
$$ 
\norm{\pd_\nu u}_{L^2(\Sigma)} \leq C\norm{f}_{H^{2,1}(\Sigma)}.
$$
The desired result follows from this by density of $\cC^\infty([0,T] \times \pd \Omega) \cap H_0^{2,1}(\Sigma)$ in $H_0^{2,1}(\Sigma)$.

\section{GO solutions}
\setcounter{equation}{0}

In this section we build GO solutions to the magnetic Schr\"odinger equation in $\Omega$. These functions are essential tools in the proof of Theorems \ref{Th1}, \ref{Th2} and \ref{Th22}. As in \cite{[Ki1]}, we take advantage of the translational invariance of $\Omega$ with respect to the longitudinal direction $x_3$, in order to adapt the method suggested by Bellassoued and Choulli in
\cite{[BC2]} for building GO solutions to the magnetic Schr\"odinger equation in a bounded domain, to the framework of the unbounded waveguide $\Omega$. Moreover, as we aim to reduce the regularity assumption imposed on the magnetic potential by the GO solutions construction method, we follow the strategy developed in \cite{[DKSU],[KLU],[KU],[Salo]} for magnetic Laplace operators, and rather build GO solutions to the Schr\"odinger equation associated with a suitable smooth approximation of the magnetic potential. 

Throughout the entire section, we consider two magnetic potentials 
$$ A_j = (A_j^{\sharp},a_{j,3}) \in W^{2,\infty}(\Omega,\R)^2 \times W^{2,\infty}(\Omega,\R),\ j=1,2, $$
and two electric potentials $q_j \in W^{1,\infty}(\Omega,\R)$, 
obeying the conditions
\bel{cond0} 
\norm{A_j}_{W^{2,\infty}(\Omega)^3} + \norm{q_j}_{W^{1,\infty}(\Omega)}\leq M,\ j=1,2,
\ee
and
\bel{cond1}
\pd_x^\alpha A_1=\pd_x^\alpha A_2\ \mbox{on}\ \pd \Omega,\ \mbox{for all}\ \alpha \in\mathbb N^3\ \mbox{such that}\ |\alpha|\leq 1. 
\ee
For $\sigma >0$, we denote by $A_{j,\sigma}^{\sharp}$ a suitable $\cC^\infty(\R^3,\R)^2 \cap W^{\infty,\infty}(\R^3,\R)^2$-approximation of $A_j^\sharp$, we shall specify in Lemma \ref{Lmol}, below.
We seek solutions $u_{j,\sigma}$ to the magnetic Schr\"odinger equation of \eqref{1.1}, where $(A_j,q_j)$ is substituted for $(A,q)$, of the form
\bel{GO} 
u_{j,\sigma}(t,x',x_3) :=\Phi_{j}(2\sigma t,x) b_{j,\sigma}(2\sigma t,x) e^{i\sigma(x'\cdot\theta-\sigma t)}
+\psi_{j,\sigma}(t,x),\ 
t \in \R,\ x=(x',x_3) \in \omega \times \R.
\ee
Here, $\theta \in \mathbb S^{1}:=\{y\in\R^2:\ |y|=1\}$ is fixed,
\bel{def-b}
b_{j,\sigma}(t,x) : =\exp \para{-i\int_0^t\theta\cdot A^\sharp_{j,\sigma}(x'-s\theta,x_3) \d s},\ 
t \in \R,\ x=(x',x_3) \in \omega \times \R, 
\ee
$\Phi_{j}$ is a solution to the following transport equation
\bel{eq-transport}
\para{\pd_t + \theta\cdot\nablas}\Phi_{j}=0\ \mbox{in}\ \R \times \Omega,
\ee
and we imposed that the remainder term $\psi_{j,\sigma} \in L^2(Q)$ scales at best like $\sigma^{-1/2}$ when $\sigma$ is large, i.e.
\bel{asy-cond}
\lim_{\sigma \to +\infty} \sigma^{1/2} \norm{\psi_{j,\sigma}}_{L^2(Q)} = 0.
\ee
Such a construction requires that $A^\sharp_{j,\sigma}$ be sufficiently close to $A_j^\sharp$, as will appear in the coming subsection.

\subsection{Magnetic potential mollification}

\label{sec-mol}

We aim to define a suitable smooth approximation 
$$A_{j,\sigma}^{\sharp}\in \cC^\infty(\R^3,\R)^2 \cap W^{\infty,\infty}(\R^3,\R)^2 $$ 
of $A_{j}^{\sharp}=(a_{1,j},a_{2,j})$, for $j=1,2$. This preliminarily requires that $A_{j}^{\sharp}$ be appropriately extended to a larger domain than $\Omega$, as follows.

\begin{lemma}
\label{ll1m}
Let $A_{j}^{\sharp}$, for $j=1,2$, be in $W^{2,\infty}(\Omega,\R)^2$ and satisfy \eqref{cond1}. Let $\tilde{\omega}$ be a smooth open bounded subset of $\R^2$, containing $\overline{\omega}$. Then, there exist two potentials $\tilde{A}_1^{\sharp}$ and $\tilde{A}_2^{\sharp}$ in $W^{2,\infty}(\R^3,\R)^2$, both of them being supported in $\tilde{\Omega}:=\tilde{\omega}\times\R$, such that we have
\bel{ll1a} 
\tilde{A}_j^\sharp=A_j^\sharp\ \mbox{in}\ \Omega,\ \mbox{for}\ j=1,2,\ \mbox{and}\ \tilde{A}_1^\sharp = \tilde{A}_2^\sharp\ \mbox{in}\ \tilde{\Omega} \setminus \Omega.
\ee
Moreover, the two estimates
\bel{ll1b}
\norm{\tilde{A}_j^\sharp}_{W^{2,\infty}(\R^3)^2} \leq C \max \left( \norm{A_1^\sharp}_{W^{2,\infty}(\Omega)^2}, \norm{A_2^\sharp}_{W^{2,\infty}(\Omega)^2} \right),\ j=1,2,
\ee
hold for some constant $C>0$, depending only on $\omega$ and $\tilde{\omega}$.
\end{lemma}
\begin{proof} 
By \cite[Section 3, Theorem 5]{[St]} and \cite[Lemma 2.7]{[KPS2]}, there exists $\tilde{A}_1^\sharp\in W^{2,\infty}(\R^3,\R)^2$, such that
$\tilde{A}_1^\sharp=A_1^\sharp$ in $\Omega$, and \eqref{ll1b} holds true for $j=1$. Then, upon possibly substituting $\chi \tilde{A}_1^\sharp$ for $\tilde{A}_1^\sharp$, where $\chi \in C_0^\infty(\R^3,\R)$ is supported in $\tilde{\Omega}$ and verifies $\chi(x)=1$ for all $x \in \Omega$, we may assume that $\tilde{A}_1^\sharp$ is supported in $\tilde{\Omega}$ as well.

Next, we put
\bel{def-a2ts}
\tilde{A}_2^\sharp(x) :=
\left\{
\begin{array}{ll} 
A_2^\sharp(x), & \mbox{if}\ x \in \Omega, \\ \tilde{A}_1^\sharp(x), & \mbox{if}\ x \in \R^3 \setminus \Omega. 
\end{array}
\right.
\ee
Then, it is clear from \eqref{cond1} that $\tilde{A}_2^\sharp \in W^{2,\infty}(\R^3,\R)^2$ and that it satisfies \eqref{ll1b} with $j=2$.
\end{proof}

Having seen this, we define for each $\sigma>0$ the smooth approximation $a_{\sigma}\in \cC^\infty(\R^3,\R) \cap W^{\infty,\infty}(\R^3,\R)$ of a function $\tilde{a} \in W^{2,\infty}(\R^3,\R)$, supported in $\tilde{\Omega}$, by 
\bel{def-as}
a_{\sigma}(x):=\int_{\R^3} \chi_\sigma(x-y) \left( \tilde{a}(y) + (x-y) \cdot \nabla \tilde{a}(y) \right) \d y,\ x \in \R^3.
\ee
Here we have set $\chi_\sigma(x):=\sigma \chi(\sigma^{1/3} x)$ for all $x \in \R^3$, where $\chi \in \cC_0^\infty(\R^3,\R_+)$ is such that  
$$ \supp\ \chi \subset \{x \in \R^3;\ |x|\leq 1\}\ \mbox{and}\ \int_{\R^3}\chi(x) \d x=1. $$
This terminology is justified by the fact that $a_\sigma$ gets closer to $\tilde{a}$ as the parameter $\sigma$ becomes larger, as can be seen from the following result.

\begin{lemma}
\label{Lmol0}
Let $\tilde{a} \in W^{2,\infty}(\R^3,\R)$ be supported in $\tilde{\Omega}$ and satisfy $\| \tilde{a} \|_{W^{2,\infty}(\R^3)} \leq M$, for some $M>0$.
Then, there exists a constant $C>0$, depending only on $\omega$, $\tilde{\omega}$, and $M$, such that for all $\sigma >0$, we have 
\bel{a-mol0}
\norm{a_{\sigma}-\tilde{a}}_{W^{k,\infty}(\R^3)}\leq C \sigma^{(k-2)/3},\ k=0,1,
\ee
where $W^{0,\infty}(\Omega)$ stands for $L^\infty(\Omega)$, and
\bel{a-mol2}
\norm{a_{\sigma}}_{W^{k,\infty}(\R^3)} \leq C\sigma^{(k-2)/3},\ k \geq 2.
\ee
\end{lemma}
\begin{proof}
We only establish \eqref{a-mol0}, the estimate \eqref{a-mol2} being obtained with similar arguments.
For $x \in \R^3$ fixed, we make the change of variable $\eta=\sigma^{1/3}(x-y)$ in \eqref{def-as}. We get
\bel{a11}
a_{\sigma}(x)=\int_{\R^3} \chi(\eta) \tilde{a}(x-\sigma^{-1/3}\eta) \d \eta + \sigma^{-1/3} \int_{\R^3} \chi(\eta) \left( \eta \cdot \nabla \tilde{a}(x-\sigma^{-1/3}\eta) \right) 
\d \eta.
\ee
On the other hand, we have
\beas
\int_{\R^3} \chi(\eta) \tilde{a}(x-\sigma^{-1/3}\eta) \d \eta - \tilde{a}(x) & = & \int_{\R^3} \chi(\eta) \left( \tilde{a}(x-\sigma^{-1/3}\eta)-\tilde{a}(x) \right) \d \eta \\
& = & -\sigma^{-1/3} \int_{\R^3} \chi(\eta) \left( \int_0^1 \eta \cdot \nabla \tilde{a}(x-s \sigma^{-1/3} \eta) \d s \right) \d \eta,
\eeas
so we infer from \eqref{a11} that
\bel{lmol1}
a_{\sigma}(x) - \tilde{a}(x) =
\sigma^{-1/3}\int_{\R^3}\chi(\eta)\left( \int_0^1 \eta \cdot \left( \nabla \tilde{a}(x-\sigma^{-1/3}\eta) - \nabla \tilde{a}(x- s \sigma^{-1/3}\eta) \right) \d s \right) \d \eta.
\ee
By the Sobolev embedding theorem (see e.g.  \cite[Theorem 1.4.4.1]{[Gr]} and \cite[Lemma 3.13]{[Ta]}), we know that
$\tilde{a} \in \cC^{1,1}(\R^3)$ satisfies the estimate $\norm{\tilde{a}}_{\cC^{1,1}(\R^3)} \leq C \norm{\tilde{a}}_{W^{2,\infty}(\R^3)}$, where $C>0$ is independent of $\tilde{a}$.
From this and \eqref{lmol1}, it then follows that
$$
| a_{\sigma}(x)-\tilde{a}(x) |\leq C \norm{\tilde{a}}_{W^{2,\infty}(\R^3)} \left( \int_{\R^3} \chi(\eta) |\eta|^2 \d \eta \right) \sigma^{-2/3},
$$
which, together with the estimate $\| \tilde{a} \|_{W^{2,\infty}(\R^3)} \leq M$, yields \eqref{a-mol0} with $k=0$. Further, upon differentiating \eqref{lmol1} with respect to $x_i$, for $i=1,2,3$, and upper bounding the integrand
function $(\eta, s) \mapsto \nabla \pd_i \tilde{a}(x-\sigma^{-1/3}\eta) - \nabla \pd_i \tilde{a}(x - s \sigma^{-1/3}\eta)$ by $2 \| \tilde{a} \|_{W^{2,\infty}(\R^3)}$, uniformly over $\R^3 \times (0,1)$, we obtain \eqref{a-mol0} for $k=1$. 
\end{proof}

We notice for further use from \eqref{def-as} and the expression of $\chi_\sigma$, that
\beas
a_\sigma(x) & = & \int_{\R^3} \left( \chi_{\sigma}(x-y) - \nabla \cdot \left( (x-y) \chi_\sigma(x-y) \right)  \right) \tilde{a}(y) \d y \\
& = & \int_{\R^3} \left( 4 \sigma \chi(\sigma^{1 \slash 3}(x-y))  + \sigma^{4 \slash 3} (x-y) \cdot \nabla \chi(\sigma^{1 \slash 3}(x-y)) \right) \tilde{a}(y)\d y,\ x \in \R^3.
\eeas
Making the change of variable $z=\sigma^{1 \slash 3}(x-y)$ in the above integral, we find that
$$ a_\sigma(x) = \int_{\R^3} \left( 4 \chi(z) + z \cdot \nabla \chi(z) \right) \tilde{a}(\sigma^{-1 \slash 3} z - x) \d z,\ x \in \R^3. $$
Since $\chi$ is compactly supported in $\R^3$, this entails that
\bel{a40}
\norm{a_{\sigma}}_{L^\infty(\R^3)} \leq C \norm{\tilde{a}}_{L^\infty(\R^3)},\ \sigma >0, 
\ee
where the constant $C>0$ depends only on $\chi$.

Let $\tilde{A}_j^\sharp=(\tilde{a}_{1,j},\tilde{a}_{2,j})$, $j=1,2$, be given by Lemma \ref{ll1m}. With reference to \eqref{def-as}, we define the smooth magnetic potentials
$$ A_{j,\sigma}^{\sharp}=(a_{1,j,\sigma},a_{2,j,\sigma}) \in \cC^\infty(\R^3,\R)^2 \cap W^{\infty,\infty}(\R^3,\R)^2, $$
by setting
\bel{def-aijsigma}
a_{i,j,\sigma}(x):=\int_{\R^3} \chi_\sigma(x-y) \left( \tilde{a}_{i,j}(y) + (x-y) \cdot \nabla \tilde{a}_{i,j}(y) \right) \d y,\ x \in \R^3,\ i,j=1,2.
\ee
Thus, applying Lemma \ref{Lmol0} with $\tilde{a}=\tilde{a}_{i,j}$ for $i,j=1,2$, we obtain the following result.

\begin{lemma}
\label{Lmol}
For $j=1,2$, let $A_{j}^{\sharp}$ be the same as in Lemma \ref{ll1m}, and satisfy \eqref{cond0}.
Then, there exists a constant $C>0$, depending only on $\omega$ and $M$, such that for each $j=1,2$, and all $\sigma >0$, we have
\bel{mol1}
\norm{A_{j,\sigma}^{\sharp}-A_{j}^{\sharp}}_{W^{k,\infty}(\Omega)^2} \leq \norm{A_{j,\sigma}^{\sharp}-\tilde{A}_{j}^{\sharp}}_{W^{k,\infty}(\R^3)^2} \leq C \sigma^{(k-2)/3},\ k=0,1,
\ee
where $\tilde{A}_{j}^{\sharp}$ is given by Lemma \ref{ll1m}, and
\bel{mol2}
\norm{A_{j,\sigma}^{\sharp}}_{W^{k,\infty}(\R^3)^2} \leq C \sigma^{(k-2)/3} \norm{\tilde{A}_{j}^{\sharp}}_{W^{2,\infty}(\R^3)} \leq C\sigma^{(k-2)/3},\ k \geq 2.
\ee
\end{lemma}

For further use, we notice from \eqref{def-b} and from \eqref{mol2} with $k=2$, that the following estimate
\bel{a1b}
\| b_{j,\sigma} \|_{W^{2,\infty}(\R \times \Omega)} + \| \pd_t b_{j,\sigma} \|_{W^{2,\infty}(\R \times \Omega)} \leq C,\ j=1,2,
\ee
holds uniformly in $\sigma >0$, for some constant $C>0$ which is independent of $\sigma$. Moreover, it can be checked through direct calculation from \eqref{def-b}, that
\beas
\theta \cdot \nablas b_{j,\sigma}(t,x) 
&  = & -i \left( \sum_{m=1}^2 \theta_m \int_0^t \sum_{k=1}^2 \theta_k \pd_{x_k} a_{j,m,\sigma}(x'-s\theta,x_3) \d s \right) b_j(t,x) \\
& = & i \left( \sum_{k=1}^2 \theta_k \int_0^t \frac{\d}{\d s} a_{j,k,\sigma}(x'-s\theta,x_3) \d s \right) b_j(t,x) \\
& = & i \left( \theta \cdot A^\sharp_{j,\sigma} (x'-t\theta,x_3 ) -\theta \cdot A^\sharp_{j,\sigma}(x',x_3) \right) b_j(t,x),  
\eeas
for all $(t,x) \in Q$, from where we see that $b_{j,\sigma}$ is solution to the transport equation
\bel{eq-trb}
( \pd_t + \theta \cdot \nablas +i \theta\cdot A_{j,\sigma}^\sharp ) b_{j,\sigma}=0,\ \mbox{in}\ Q,\ \sigma \in \R_+^*,\ j=1,2.
\ee

We turn now to building suitable GO solutions to the magnetic Schr\"odinger equation of \eqref{1.1}.

\subsection{Building GO solutions to magnetic Schr\"odinger equations}
For $j=1,2$, we seek GO solutions to the magnetic Schr\"odinger equation of \eqref{1.1} with $(A,q)$ replaced by $(A_j,q_j)$, obeying \eqref{GO}--\eqref{asy-cond}, where the function $A_{j,\sigma}^{\sharp}$, appearing in \eqref{def-b}, is the smooth magnetic potential described by Lemma \ref{Lmol}. This requires that the functions $\Phi_j$, appearing in \ref{GO}, be preliminarily defined more explicitly.
To do that, we set $B(0,r) := \{ x' \in \R^2;\ |x'|< r\}$ for all $r>0$, and take $R>1$ so large that $\overline{\tilde{\omega}} \subset B(0,R-1)$,
where $\tilde{\omega}$ is the same as in Lemma \ref{ll1m}. 
Next, we pick $\phi_j \in \cC_0^\infty(\R^3)$, such that
\bel{def-phi}
\supp\ \phi_j(\cdot,x_3) \subset \cD_R:= B(0,R+1) \backslash \overline{B(0,R)},\ x_3 \in \R,
\ee
and put
\bel{def-Phi}
\Phi_j(t,x):=\phi_j(x'-t \theta ,x_3),\ (t,x) \in \R \times \R^3. 
\ee
It is apparent from \eqref{def-phi} and the embedding $\omega \subset B(0,R-1)$, that 
\bel{supp-phi1}
\supp\ \phi_j(\cdot,x_3) \cap \omega = \emptyset,\ x_3 \in \R,
\ee
and from \eqref{def-Phi}, that $\Phi$ is solution to the transport equation \eqref{eq-transport}.

In the sequel, we choose $\sigma > \sigma_* :=(R+1) \slash T$, in such a way that
\bel{supp-phi2}
\supp\ \Phi_j(\pm 2 \sigma t,\cdot,x_3) \cap \omega = \supp\ \phi_j(\cdot \mp 2 \sigma t \theta,x_3)  \cap \omega = \emptyset,\ (t,x_3) \in [T,+\infty) \times \R.
\ee
Notice that upon possibly enlarging $R$, we may assume that $\sigma_* \geq 1$, which will always be the case in the remaining part of this text. 

Next, for $k \in \mathbb{N}$, we introduce the following subspace of $H^k(\R^3)$,
$$ \cH_\theta^k := \{ \phi \in H^k(\R^3);\ \theta \cdot \nabla_{x'} \cdot \phi \in  H^k(\R^3)\ \mbox{and}\ \supp\ \phi(\cdot,x_3) \subset \cD_R\ \mbox{for a.e.}\ x_3 \in \R \}, $$
endowed with the norm
\bel{def-nrm}
N_{k,\theta}(\phi) := \norm{\phi}_{H^k(\R^3)}+ \norm{\theta \cdot \nabla_{x'} \phi}_{H^k(\R^3)},\ \phi \in \cH_\theta^2.
\ee
For notational simplicity, we put
\bel{def-snrm}
\cN_{\theta,\sigma}(\phi):= N_{2,\theta}(\phi)+ \sigma^{1 \slash 3} N_{0,\theta}(\phi).
\ee
The coming statement claims existence of GO solutions $u_{j,\sigma}$, expressed by \eqref{GO}, with 
$L^2(0,T;H^k(\Omega))$-norm of correction term $\psi_{j,\sigma}$ bounded by $\cN_{\theta,\sigma}(\phi_j) \slash \sigma^{1-k}$ for $k=0,1$.

\begin{Prop}
\label{pr-4.1}
Let $M>0$, and let $A_{j} \in W^{2,\infty}(\Omega,\R^3)$ and $q_j \in W^{1,\infty}(\Omega,\R)$, $j=1,2$, satisfy \eqref{cond0}-\eqref{cond1}. Then, for all $\sigma > \sigma_*$, there exists 
$u_{j,\sigma} \in \cC^1([0,T],L^2(\Omega)) \cap \cC([0,T],H^2(\Omega))$ obeying \eqref{GO}-\eqref{asy-cond}, where $\Phi_j$ is defined by \eqref{def-phi}-\eqref{def-Phi}, such that
we have
$$
\para{i \pd_t +\Delta_{A_j} + q_j} u_{j,\sigma}=0\ \mbox{in}\ Q,
$$
and the correction term satisfies $\psi_{j,\sigma}=0$ on $\Sigma$, for $j=1,2$, and $\psi_{1,\sigma}(T,\cdot)=\psi_{2,\sigma}(0,\cdot)=0$ in $\Omega$.\\
Moreover, the following estimate
\bel{4.6}
\sigma \norm{\psi_{j,\sigma}}_{L^2(Q)}+\norm{\nabla \psi_{j,\sigma}}_{L^2(Q)^3}
\leq  C \cN_{\theta,\sigma}(\phi_j),\ j=1,2,
\ee
holds for some constant $C>0$ depending only on $T$, $\omega$, and $M$, where the function $\phi_j \in \cC_0^\infty(\R^3)$ fulfills \eqref{def-phi}. 
\end{Prop}

\begin{proof}
We prove the result for $j=2$, the case $j=1$ being obtained in the same way.

In light of \eqref{GO}--\eqref{eq-transport} and the identity $(i \pd_t + \Delta_{A_2} +q_2) u_{2,\sigma}=0$ imposed on $ u_{2,\sigma}$ in $Q$, we seek
a solution $\psi_{2,\sigma}$ to the following IBVP
\bel{4.7}
\left\{
\begin{array}{ll}
\para{i \pd_t + \Delta_{A_2}+q_2} \psi_{2,\sigma}=g_\sigma, & \mbox{in}\ Q, \\
\psi_2(0,\cdot)=0, & \mbox{in}\ \Omega, \\
\psi_2=0, & \mbox{on}\ \Sigma,
\end{array}
\right.
\ee
where
$$
g_\sigma :=-\para{i\partial_t+\Delta_{A_2}+q_2}\para{w_\sigma \varphi_\sigma},
$$
with
\bel{a12}
w_\sigma(t,x') :=e^{i\sigma (x'\cdot \theta -\sigma t)}\ \mbox{and}\ \varphi_\sigma(t,x):=\vartheta_\sigma(2 \sigma t , x),\ \mbox{where}\ 
\vartheta_\sigma:= \Phi_2 b_{2,\sigma}.
\ee

Next, taking into account that $(i \pd_t +\Delta_{A_2}+q_2) w_\sigma = (i \nabla \cdot A_2 -|A_2|^2  - 2 \sigma \theta \cdot A_2^\sharp + q_2) w_\sigma$, and recalling from
\eqref{eq-transport} and \eqref{eq-trb} that
$i(\pd_t +2 \sigma \theta \cdot \nabla_{x'})\varphi_\sigma = 2\sigma \theta \cdot A^\sharp_{2,\sigma} \varphi_\sigma$, we get by straightforward computations that
\bel{a13}
g_\sigma(t,x) = -w_\sigma(t,x) \sum_{m=0,1} g_{m,\sigma}(2\sigma t,x),\ \mbox{with}\
g_{0,\sigma} := (\Delta _{A_2}+q_2) \vartheta_\sigma,\ 
g_{1,\sigma}:= 2 \sigma \theta \cdot (A_{2,\sigma}^\sharp-A_2^\sharp) \vartheta_\sigma.
\ee

As $g_{\sigma} \in W^{1,1}(0,T; L^2(\Omega ))$, by \eqref{def-phi}-\eqref{def-Phi}, we know from Lemma \ref{ll2} that \eqref{4.7} admits a unique solution 
$\psi_{2,\sigma}\in \cC^1([0,T],L^2(\Omega))\cap \cC([0,T],H^1_0(\Omega)\cap H^2(\Omega))$.
Moreover, since
$$ \psi_{2,\sigma}(t,x) =-i \int_0^t e^{-i(t-s) \sH_{A_2,q_2}} g_{\sigma}(s,x) \d s,\ (t,x) \in Q, $$
where $\sH_{A_2,q_2}$ is the self-adjoint operator acting in $L^2(\Omega)$, which is defined in Subsection \ref{sec-elliptic}, we have
$$
\| \psi_{2,\sigma}(t,\cdot) \|_{L^2(\Omega)}  \leq \int_0^t \| e^{-i(t-s) \sH_{A_2,q_2}} g_{\sigma}(s,\cdot) \|_{L^2(\Omega)} \d s \leq \| g_{\sigma} \|_{L^1(0,T;L^2(\Omega))}, 
$$
uniformly in $t \in (0,T)$. This entails $\norm{\psi_{2,\sigma}}_{L^2(Q)} \leq T^{1/2} \norm{g_{\sigma}}_{L^1(0,T;L^2(\Omega))}$, which together with 
\eqref{a13}, yields
\bel{l4.1a}
\norm{\psi_{2,\sigma}}_{L^2(Q)} \leq T^{1/2} \sum_{m=0,1} \int_0^T \norm{g_{m,\sigma}(2\sigma t,\cdot)}_{L^2(\Omega)} \d t \leq \sigma^{-1} T^{1/2} \sum_{m=0,1} \norm{g_{m,\sigma}}_{L^1(\R,L^2(\Omega))}.
\ee
We are left with the task of bounding each term $\norm{g_{m,\sigma}}_{L^1(\R,L^2(\Omega))}$, for $m=0,1$, separately. We start with $m=0$, and obtain
\bea
\norm{g_{0,\sigma}}_{L^1(\R,L^2(\Omega))} & = & \int_\R \norm{(\Delta_{A_2}+q_2)(\Phi_2 b_{2,\sigma})(s,\cdot)}_{L^2(\Omega)} \d s \nonumber \\
& \leq & C \norm{b_{2,\sigma}}_{W^{2,\infty}(\R \times \Omega)} \norm{\phi_2}_{H^2(\R^3)} \leq C \norm{\phi_2}_{H^2(\R^3)}, \label{a14}
\eea
by combining estimate \eqref{a1b} with definitions \eqref{def-phi}-\eqref{def-Phi} and \eqref{a13}.
Next, applying \eqref{mol1} with $k=0$, we get that
\bel{a15}
\norm{g_{1,\sigma}}_{L^1(\R,L^2(\Omega))} \leq C \sigma \norm{A_{2,\sigma}^\sharp-A_2^\sharp}_{L^\infty(\Omega)} \norm{\phi_2}_{L^2(\R^3)} \leq C \sigma^{1 \slash 3} \norm{\phi_2}_{L^2(\R^3)},
\ee
Putting \eqref{l4.1a}--\eqref{a15} together, and recalling \eqref{def-nrm}, we find that
\bel{l4.1c}
\sigma \norm{\psi_{2,\sigma}}_{L^2(Q)} \leq C \left( \norm{\phi_2}_{H^2(\R^3)} + \sigma^{1 \slash 3} \norm{\phi_2}_{L^2(\R^3)} \right) \leq C \left( N_{2,\theta}(\phi_2) + \sigma^{1 \slash 3} N_{0,\theta}(\phi_2) \right).
\ee
It remains to bound $\norm{\nabla \psi_{2,\sigma}}_{L^1(\R,L^2(\Omega))}$ from above. To do that, we apply \cite[Lemma 3.2]{[BD]}, which is permitted since $g_{\sigma}(0,\cdot)=0$, with $\varepsilon=\sigma^{-1}$, getting
\bea
\norm{\nabla \psi_{2,\sigma}(t,\cdot)}_{L^2(\Omega)}  & \leq & C \left( \sigma \norm{g_{\sigma}}_{L^1(0,T;L^2(\Omega))} + \sigma^{-1} \norm{\pd_t g_{\sigma}}_{L^1(0,T;L^2(\Omega))}\right) \nonumber \\
& \leq & C \sum_{m=0,1} \left( \sigma \int_0^T \norm{g_{m,\sigma}(2 \sigma t,\cdot)}_{L^2(\Omega)} \d t
+ \int_0^T \norm{\pd_t g_{m,\sigma}(2\sigma t,\cdot)}_{L^2(\Omega)} \d t \right) \nonumber \\
& \leq & C \sum_{m=0,1} \left( \| g_{m,\sigma} \|_{L^1(\R,L^2(\Omega))} + \| \pd_t g_{m,\sigma} \|_{L^1(\R,L^2(\Omega))} \right), \label{a16}
\eea
for every $t \in (0,T)$, according to \eqref{a12}-\eqref{a13}. Further, as we have
$$ \pd_t g_{0,\sigma}(t,x) = -(\Delta_{A_2}+q_2)  \theta \cdot \left( \nabla_{x'} \phi_2+ i A_{2,\sigma}^\sharp \phi_2 \right)(x'-t \theta,x_3) b_{2,\sigma}(t,x), $$
for a.e. $(t,x) \in \R \times \Omega$, by direct computation, we obtain
\bel{a16b}
\| \pd_t g_{0,\sigma} \|_{L^1(\R,L^2(\Omega))} \leq C N_{2,\theta}(\phi_2),
\ee
from \eqref{def-b}, \eqref{mol2} with $k=2$, \eqref{a1b} with $j=2$, \eqref{def-phi}-\eqref{def-Phi} and \eqref{a12}-\eqref{a13}. 
Similarly, as
$$ \pd_t g_{1,\sigma}(t,x) = -2 \sigma \theta \cdot (A_2^\sharp-A^\sharp_{2,\sigma})(x)  \theta \cdot \left( \nabla_{x'} \phi_2 + i A_{2,\sigma}^\sharp  \phi_2 \right)(x'-t \theta,x_3) b_{2,\sigma}(t,x), $$
for a.e. $(t,x) \in \R \times \Omega$, we find that
\bel{a16c}
\| \pd_t g_{1,\sigma} \|_{L^1(\R,L^2(\Omega))}
\leq C \sigma \norm{A_2^\sharp-A^\sharp_{2,\sigma}}_{L^\infty(\Omega)} N_{0,\theta}(\phi_2)
\leq C \sigma^{1/3}  N_{0,\theta}(\phi_2),
\ee
according to \eqref{mol1} with $j=2$ and $k=0$. Thus, we infer from \eqref{a14}-\eqref{a15} and \eqref{a16}--\eqref{a16c}, that 
$$
\norm{\nabla \psi_{2,\sigma}(t,\cdot)}_{L^2(\Omega)^3}\leq  C \left( N_{2,\theta}(\phi_2) + \sigma^{1 \slash 3} N_{0,\theta}(\phi_2) \right),\ t \in (0,T),\ \sigma > \sigma_*.
$$
This and \eqref{l4.1c} yield \eqref{4.6} with $j=2$, upon recalling the definition \eqref{def-snrm}.
\end{proof}

Let us now establish for further use that we may substitute $\sigma^{-1 \slash 6} u_{j,\sigma}$ for $\psi_{j,\sigma}$ in the estimate \eqref{4.6}.

\begin{Corollary}
\label{cor-1}
For $j=1,2$, let $q_j$, $A_j$, $\phi_j$, and $u_{j,\sigma}$, be the same as in Proposition \ref{pr-4.1}.
Then, there exists a constant $C>0$, depending only on $T$, $\omega$ and $M$, such that
the estimate
\bel{r0}
\sigma \norm{u_{j,\sigma}}_{L^2(Q)}+\norm{\nabla u_{j,\sigma}}_{L^2(Q)^3}
\leq  C \sigma^{1 \slash 6} \cN_{\theta,\sigma}(\phi_j),\ j=1,2,
\ee
holds for all $\sigma > \sigma_*$.
\end{Corollary}
\begin{proof}
Notice from \eqref{def-Phi} and \eqref{supp-phi2} that
$$ \int_0^T \| \Phi_j(2 \sigma t, \cdot) \|_{H^k(\Omega)}^2 \d t  = \int_0^{+\infty} \| \Phi_j(2 \sigma t, \cdot) \|_{H^k(\Omega)}^2 \d t 
=  (2 \sigma)^{-1} \int_0^{2R} \| \Phi_j(s, \cdot) \|_{H^k(\Omega)}^2 \d s, $$
so we have
\bel{55}
\| \Phi_j(2 \sigma \cdot, \cdot) \|_{L^2(0,T;H^k(\Omega))} \leq R^{1 \slash 2} \sigma^{-1 \slash 2} \norm{\phi_j}_{H^k(\R^3)},\ j=1,2,\ k \in \mathbb{N}. 
\ee
From this, \eqref{GO}, \eqref{a1b} and \eqref{def-nrm}--\eqref{4.6}, it follows for each $j=1,2$, that
\beas
\| u_{j,\sigma} \|_{L^2(Q)} & \leq & \| b_{j,\sigma} \|_{L^\infty(\R \times \Omega)} \| \Phi_j(2 \sigma \cdot, \cdot) \|_{L^2(Q)} + \| \psi_{j,\sigma} \|_{L^2(Q)} \nonumber \\
& \leq & C \left( \sigma^{-1 \slash 2} \| \phi_j \|_{L^2(\R^3)}  + \sigma^{-1} \cN_{\theta,\sigma}(\phi_j) \right) 
\leq C \sigma^{-5 \slash 6} \cN_{\theta,\sigma}(\phi_j),
\eeas
and
\beas
& & \| \nabla u_{j,\sigma} \|_{L^2(Q)^3} \\
& \leq & \| b_{j,\sigma} \|_{W^{1,\infty}(\R \times \Omega)} \left( \sigma \| \Phi_j(2 \sigma \cdot, \cdot) \|_{L^2(Q)} + \| \Phi_j(2 \sigma \cdot, \cdot) \|_{L^2(0,T;H^1(\Omega))} \right) + 
\| \nabla \psi_{j,\sigma} \|_{L^2(Q)^3} \nonumber \\
& \leq & C \left( \sigma^{1 \slash 2} \| \phi_j \|_{L^2(\R^3)}  + \sigma^{-1 \slash 2} \| \phi_j \|_{H^1(\R^3)} + \cN_{\theta,\sigma}(\phi_j) \right) 
\leq C \sigma^{1 \slash 6} \cN_{\theta,\sigma}(\phi_j),
\eeas
which yields \eqref{r0}.
\end{proof}

In the coming subsection we probe the medium with the GO solutions described in Proposition \ref{pr-4.1} in order to upper bound the transverse magnetic potential in terms of suitable norm of the DN map.

\subsection{Probing the medium with GO solutions}
Let us introduce
\bel{def-Atilde-Asigma}
\tilde{A}^\sharp:=\tilde{A}_2^\sharp-\tilde{A}_1^\sharp,\mbox{and}\ A_\sigma^\sharp:=A_{2,\sigma}^\sharp-A_{1,\sigma}^\sharp,\ \sigma >0,
\ee
where the functions $\tilde{A}_j^\sharp$ and $A_{j,\sigma}^\sharp$, $j=1,2$, are defined in Lemma \ref{ll1m} and Lemma \ref{Lmol}, respectively.
Evidently, $\tilde{A}^\sharp$ is the function $A_2^\sharp - A_1^\sharp$, extended by zero outside $\Omega$, and we have
\bel{mmol1}
\norm{A_\sigma^\sharp - \tilde{A}^\sharp}_{W^{1,\infty}(\R^3)^2} \leq \sum_{j=1,2} \norm{A_{j,\sigma}^\sharp - \tilde{A}_j^\sharp}_{W^{1,\infty}(\R^3)^2} \leq 2 C \sigma^{-1/3},\ 
\sigma >0,
\ee
from \eqref{mol1} with $k=1$. Thus, writing $A_\sigma^\sharp=(a_{1,\sigma},a_{2,\sigma})$ and $\tilde{A}^\sharp=(\tilde{a}_1,\tilde{a}_2)$, it follows readily from \eqref{def-aijsigma}
that
\bel{mol3}
a_{i,\sigma}(x) = \int_{\R^3} \chi_\sigma(x-y) \left( \tilde{a}_i(y) + (x-y) \cdot \nabla \tilde{a}_i(y) \right) \d y,\ x \in \R^3,\ i=1,2.
\ee

The main purpose of this subsection is the following technical result.


\begin{lemma}
\label{L5.1}
Let $M>0$ and $\theta \in\s^{1}$ be fixed.
For $j=1,2$, let $A_j \in W^{2,\infty}(\Omega,\R)^3$, $q_j \in W^{1,\infty}(\Omega,\R^3)$, obey \eqref{cond0}-\eqref{cond1}, and let $\phi_j$ be defined by \eqref{def-phi}.
Then, for every $\sigma > \sigma_*$, there exists a constant $C>0$, depending only on $T$, $\omega$, and $M$, such that we have
\bea
& & \sigma \abs{\int_{(0,T) \times \R^3} \theta \cdot \tilde{A}^\sharp(x) (\overline{\phi}_1 \phi_2)(x'-2 \sigma t \theta,x_3) (\overline{b_{1,\sigma}} b_{2,\sigma})(2\sigma t,x) \d x' \d x_3 \d t} \nonumber \\
& \leq & C \left( \sigma^5 \norm{\Lambda_{A_1,q_1}-\Lambda_{A_2,q_2}} +\sigma^{-5 \slash 6} \right) \cN_{\theta,\sigma}(\phi_1) \cN_{\theta,\sigma}(\phi_2), \label{5.11}
\eea
where $\norm{\cdot}$ stands for the usual norm in $\cB(H^{2,1}(\Sigma),L^2(\Sigma))$, and $\tilde{A}^\sharp$ is given by \eqref{def-Atilde-Asigma}. 
\end{lemma}

\begin{proof}
We proceed in two steps. The first step is to establish a suitable orthogonality identity for $A:=A_2-A_1$ and $V :=i \nabla \cdot A-(|A_2|^2-|A_1|^2)+q_2-q_1$, which is
the key ingredient in the derivation of the estimate \eqref{5.11}, presented in the second step.

\noindent {\it Step 1: Orthogonality identity.} We probe the system with the GO functions $u_{j,\sigma}$, $j=1,2$,
given by Proposition \ref{pr-4.1}, and recall for further use that $u_{j,\sigma} \in \cC^1([0,T],L^2(\Omega))\cap \cC([0,T],H^2(\Omega))$ is expressed by \eqref{GO} and satisfies the following equation
\bel{eq-go}
\para{i \pd_t + \Delta_{A_j} + q_j} u_{j,\sigma} =0,\ \mbox{in}\ Q.
\ee


Since $A_{2,\sigma}^\sharp \in W^{2,\infty}(\Omega)^2$ and $\phi_2 \in C_0^\infty(\R^3)$, it follows readily from \eqref{GO}-\eqref{def-b} and \eqref{def-Phi} that $u_{2,\sigma} - \psi_{2,\sigma} \in C^\infty([0,T],W^{2,\infty}(\Omega))$. Thus, $F:= -\para{i \pd_t + \Delta_{A_1} + q_1} (u_{2,\sigma} - \psi_{2,\sigma}) \in W^{1,1}(0,T;L^2(\Omega))$, and there is consequently a unique solution $z \in \cC^1([0,T],L^2(\Omega))\cap \cC([0,T], H_0^1(\Omega) \cap H^2(\Omega))$ to the IBVP
\bel{eq-upsilon}
\left\{\begin{array}{ll}
\para{i \pd_t + \Delta_{A_1} + q_1} z = F, & \mbox{in}\ Q, \\
z(0,\cdot)=0, & \mbox{in}\ \Omega, \\
z=0,  & \mbox{on}\ \Sigma ,
\end{array}
\right.
\ee
in virtue of Lemma \ref{ll2}. Further, as $(u_{2,\sigma} - \psi_{2,\sigma})(0,\cdot)=0$ in $\Omega$, by \eqref{def-Phi}-\eqref{supp-phi1}, we infer from \eqref{eq-upsilon} that 
$v:=z + u_{2,\sigma} - \psi_{2,\sigma} \in \cC^1([0,T],L^2(\Omega))\cap \cC([0,T], H^2(\Omega))$ verifies
\bel{5.2}
\left\{
\begin{array}{ll}
\para{i \pd_t +\Delta_{A_1} + q_1} v =0, & \mbox{in}\ Q,\\
v(0,\cdot)=0, & \mbox{in}\ \Omega,\\
v=f_{\sigma}, & \mbox{on}\ \Sigma,
\end{array}
\right.
\ee
where we have set
\bel{def-fsigma}
f_\sigma(t,x) := u_{2,\sigma}(t,x) = u_{2,\sigma}(t,x) - \psi_{2,\sigma}(t,x) = (\Phi_2  b_{2,\sigma})(2\sigma t,x) e^{i\sigma(x' \cdot \theta - \sigma t)},\ (t,x) \in \Sigma.
\ee
From this and Proposition \ref{pr-4.1}, it then follows that $w :=v-u_{2,\sigma}$ is the $\cC^1([0,T],L^2(\Omega))\cap \cC([0,T], H_0^1(\Omega) \cap H^2(\Omega))$-solution to
the IBVP
\bel{eq-w}
\left\{\begin{array}{ll}
\para{i \pd_t + \Delta_{A_1} + q_1}w = 2 i A \cdot \nabla u_{2,\sigma} + V u_{2,\sigma}, & \mbox{in}\ Q, \\
w(0,\cdot)=0, & \mbox{in}\ \Omega, \\
w=0,  & \mbox{on}\ \Sigma ,
\end{array}
\right.
\ee
In light of \eqref{eq-w}, we deduce from \eqref{eq-go} with $j=1$, upon applying the Green formula, that
\bea
\langle 2i A \cdot \nabla u_{2,\sigma} + V u_{2,\sigma}, u_{1,\sigma} \rangle_{L^2(Q)} & = & \langle \para{i \pd_t + \Delta_{A_1}+q_1} w , u_{1,\sigma} \rangle_{L^2(Q)} \nonumber \\
& = & \langle \para{\pd_\nu + i A_1 \cdot \nu} w , u_{1,\sigma} \rangle_{L^2(\Sigma)}. \label{5.4}
\eea
Next, taking into account that $A_1=A_2$ on $\pd \Omega$, from \eqref{cond1}, we see that
\beas 
\para{\pd_\nu + i A_1 \cdot \nu} w & = & \para{\pd_\nu + i A_1 \cdot \nu} v - \para{\pd_\nu + i A_1 \cdot \nu} u_{2,\sigma} \\
& = & \para{\pd_\nu + i A_1 \cdot \nu} v - \para{\pd_\nu + i A_2 \cdot \nu} u_{2,\sigma} \\
& = & ({\Lambda_{A_1,q_1}-\Lambda_{A_2,q_2}})f_{\sigma},
\eeas
according to \eqref{def-fsigma} and the last line of \eqref{5.2}. This and \eqref{5.4} yield the following orthogonality identity
\bel{5.5}
2i \langle A \cdot \nabla u_{2,\sigma} , u_{1,\sigma} \rangle_{L^2(Q)} + \langle  V u_{2,\sigma} , u_{1,\sigma} \rangle_{L^2(Q)} = \langle ({\Lambda_{A_1,q_1}-\Lambda_{A_2,q_2}})f_{\sigma}, g_\sigma \rangle_{L^2(\Sigma)},
\ee
with
\bel{def-gsigma}
g_\sigma (t,x) := u_{1,\sigma}(t,x) = u_{1,\sigma}(t,x) - \psi_{1,\sigma}(t,x) = (\Phi_1 b_{1,\sigma})(2 \sigma t,x) e^{i\sigma(x'\cdot\theta-2\sigma t)},\ (t,x) \in \Sigma.
\ee
Having established \eqref{5.5}, we turn now to proving the estimate \eqref{5.11}.\\

\noindent {\it Step 2: Derivation of \eqref{5.11}.}
In light of \eqref{GO}, we have
\bel{5.6}
\langle A \cdot \nabla u_{2,\sigma} , u_{1,\sigma} \rangle_{L^2(Q)} = I_{\sigma} + i \sigma \int_Q \theta \cdot A^\sharp(x) (\overline{\Phi}_1 \Phi_2 )(2\sigma
t,x) (\overline{b_{1,\sigma}} b_{2,\sigma})(2\sigma t,x) \d x \d t, 
\ee
with
\beas
I_\sigma & := &  \int_Q A \cdot \nabla(\Phi_2 b_{2,\sigma})(2\sigma t,x) \left( \overline{(\Phi_1 b_{1,\sigma})}(2\sigma t,x) + e^{i\sigma(x' \cdot \theta - \sigma t)} \overline{\psi_{1,\sigma}}(t,x) \right) \d x \d t \nonumber \\
& & +  \int_Q A \cdot \nabla \psi_{2,\sigma}(t,x) \left(e^{i\sigma(x' \cdot \theta - \sigma t)} \overline{(\Phi_1 b_{1,\sigma})}(2\sigma t,x)  + \overline{\psi_{1,\sigma}}(t,x) \right) \d x \d t \nonumber \\ 
& & + i \sigma \int_Q \theta \cdot A^\sharp(x) (\Phi_2 b_{2,\sigma})(2\sigma t,x) \overline{\psi_{1,\sigma}}(t,x) e^{i \sigma(x' \cdot \theta - \sigma t)} \d x \d t. 
\eeas
We infer from \eqref{a1b}, \eqref{4.6}, and \eqref{55}, that
$$
\abs{I_{\sigma}} \leq C \sigma^{-5/6} \cN_{\theta,\sigma}(\phi_1) \cN_{\theta,\sigma}(\phi_2),\ \sigma > \sigma_*. 
$$
Putting this together with \eqref{5.5} and \eqref{5.6}, we find that
\bea
& & \sigma \abs{\int_Q \theta \cdot A^\sharp(x) (\phi_2 \overline{\phi}_1)(x'-2\sigma t \theta,x_3) (b_{2,\sigma}\overline{b_{1,\sigma}})(2\sigma t,x) \d x' \d x_3 \d t} \nonumber \\
& \leq & C \left( \abs{\langle V u_{2,\sigma} ,u_{1,\sigma} \rangle_{L^2(Q)}} + \abs{\langle ({\Lambda_{A_1,q_1}-\Lambda_{A_2,q_2}})f_{\sigma}, g_\sigma \rangle_{L^2(\Sigma)}}
 + \sigma^{-5/6} \cN_{\theta,\sigma}(\phi_1) \cN_{\theta,\sigma}(\phi_2)  \right). \label{5.8}
\eea
Next, we notice from \eqref{r0} that
\bel{5.9}
\abs{\langle V u_{2,\sigma} , u_{1,\sigma} \rangle_{L^2(Q)}} \leq C \sigma^{-5 \slash 3} \cN_{\theta,\sigma}(\phi_1) \cN_{\theta,\sigma}(\phi_2).
\ee
Moreover, in view of \eqref{def-fsigma} and \eqref{def-gsigma}, we have
\beas
\abs{\langle ({\Lambda_{A_1,q_1}-\Lambda_{A_2,q_2}})f_{\sigma}, g_\sigma \rangle_{L^2(\Sigma)}} 
& \leq & \norm{\Lambda_{A_1,q_1}-\Lambda_{A_2,q_2}} \norm{f_\sigma}_{H^{2,1}(\Sigma)}\norm{g_\sigma}_{L^2(\Sigma)} \nonumber \\
& \leq & \norm{\Lambda_{A_1,q_1}-\Lambda_{A_2,q_2}} \norm{u_{2,\sigma}-\psi_{2,\sigma}}_{H^{2,1}(\Sigma)} \norm{u_{1,\sigma}-\psi_{1,\sigma}}_{L^2(\Sigma)}, 
\eeas
with
\beas
\norm{u_{1,\sigma}-\psi_{1,\sigma}}_{L^2(\Sigma)}  & \leq & \norm{u_{1,\sigma}-\psi_{1,\sigma}}_{L^2(0,T;H^1(\Omega))} \nonumber \\
& \leq & C \sigma \norm{\Phi_1(2 \sigma \cdot, \cdot)}_{L^2(0,T;H^1(\Omega))} \norm{b_{1,\sigma}}_{W^{1,\infty}(\R \times \Omega)} 
\nonumber \\
& \leq & C \sigma^{1 \slash 2} \cN_{\theta,\sigma}(\phi_1),
\eeas
and
\beas
\norm{u_{2,\sigma}-\psi_{2,\sigma}}_{H^{2,1}(\Sigma)} & \leq & 
C \left( \norm{u_{2,\sigma}-\psi_{2,\sigma}}_{H^2(0,T;H^1(\Omega))} + \norm{u_{2,\sigma}-\psi_{2,\sigma}}_{L^2(0,T;H^2(\Omega)} \right) \nonumber \\
& \leq & C  \sigma^5 \norm{\Phi_2(2 \sigma \cdot, \cdot)}_{L^2(0,T;H^2(\Omega))} \left(\norm{b_{2,\sigma}}_{W^{2,\infty}(\R \times \Omega)} + \norm{\pd_t b_{2,\sigma}}_{W^{2,\infty}(\R \times \Omega)} \right)  \nonumber \\
& \leq & C \sigma^{9 \slash 2} \cN_{\theta,\sigma}(\phi_2),
\eeas
according to \eqref{GO}, \eqref{a1b}, \eqref{def-nrm}, and \eqref{55}. As a consequence, we have
\bel{a20}
\abs{\langle (\Lambda_{A_1,q_1}-\Lambda_{A_2,q_2})f_{\sigma}, g_\sigma \rangle_{L^2(\Sigma)}} 
\leq C \sigma^5 \norm{\Lambda_{A_1,q_1}-\Lambda_{A_2,q_2}} \cN_{\theta,\sigma}(\phi_1) \cN_{\theta,\sigma}(\phi_2) .
\ee
This and \eqref{5.8}-\eqref{5.9} yield \eqref{5.11}. 
\end{proof}

\section{Preliminary estimates}
\setcounter{equation}{0}

\subsection{$X$-ray transform}

In this subsection we estimate the partial $X$-ray transform in $\R^3$, of the functions
\bel{6.1}
\tilde{\rho}_j(x',x_3) := \theta \cdot \frac{\pd \tilde{A}^\sharp}{\pd x_j}(x)=\sum_{i=1,2} \theta_i \frac{\pd \tilde{a}_i}{\pd x_j}(x),\ x \in \R^3,\ j=1,2,3,
\ee
in terms of the DN map. We recall that the partial $X$-ray transform in the direction $\theta \in \mathbb{S}^{1}$, of a function 
\bel{def-X}
f \in \sX : = \{ \varphi \in L_{\mathrm{loc}}^1(\R^3);\ x' \mapsto \varphi(x',x_3) \in L^1(\R^2)\ \mbox{for a.e.}\ x_3 \in \R \},
\ee
is defined as
\bel{def-P}
\cP(f)(\theta,x',x_3):= \int_\R f(x' + s \theta , x_3) \d s,\ x'\in \R^2,\ x_3 \in \R.
\ee


The $X$-ray transform stability estimate is as follows.

\begin{lemma}
\label{L6.2}
Let $M>0$, and let $A_j$ and $q_j$, for $j=1,2$, be as in Proposition \ref{pr-4.1}.
Then, there exists a constant $C>0$, depending only on $T$, $\omega$ and $M$, such that for all
$\theta \in \s^{1}$, all $\xi' \in \R^2$, and all $\phi \in C_0^\infty(\R^3)$ satisfying
$\supp\ \phi(\cdot,x_3) \subset \cD_R^-(\theta) := \{ x' \in \cD_R,\ x' \cdot \theta \leq 0 \}$ for every $x_3 \in \R$,
the estimate
\bea
& & \abs{\int_{\R^3} \phi^2(x) \cP(\tilde{\rho}_j)(\theta, x',x_3)
\exp\para{-i\int_\R \theta \cdot A_\sigma^\sharp (x'+s\theta,x_3) \d s} \d x} \nonumber \\
& \leq & C \left( \sigma^5 \norm{\Lambda_{A_1,q_1}-\Lambda_{A_2,q_2}} +\sigma^{-5/6} \right) \cN_{\theta,\sigma}(\phi) \cN_{\theta,\sigma}(\pd_{x_j} \phi),
\label{66.2}
\eea
holds uniformly in $\sigma >\sigma_*$ and $j=1,2,3$. 
\end{lemma}

\begin{proof}
Let $\phi_j \in \cC_0^\infty(\R^3)$, $j=1,2$, be supported in $\cD_R \times \R$.
Bearing in mind that $\overline{\tilde{\omega}} \subset B(0,R-1)$, we infer from \eqref{mol3} that $\tilde{A}^\sharp$ and $A_\sigma^\sharp$ are both supported in $B(0,R) \times \R$. Further, as $|x'- 2 \sigma t \theta| > 2 \sigma_* T - R > R+1$ for all $x'\in B(0,R)$ and $t>T$, we see that
$$
\tilde{A}^\sharp(x) (\overline{\phi}_1 \phi_2)(x'-2\sigma t\theta,x_3) = A_\sigma^\sharp(x)(\overline{\phi}_1 \phi_2)(x'- 2 \sigma t \theta,x_3)=0,\ x=(x',x_3) \in\R^3,\ t >T, 
$$
As a consequence we have
\beas
& & \int_0^T \int_{\R^3} \theta \cdot \left( \tilde{A}^\sharp(x)-A_\sigma^\sharp(x) \right) (\overline{\phi_1} \phi_2)(x'-2\sigma t\theta,x_3) (\overline{b_{1,\sigma}} b_{2,\sigma})(2\sigma t,x) \d x \d t \\
& = & \int_0^{+\infty} \int_{\R^3} \theta \cdot \left( \tilde{A}^\sharp(x)-A_\sigma^\sharp(x) \right) (\overline{\phi_1} \phi_2)(x'-2\sigma t\theta,x_3) (\overline{b_{1,\sigma}} b_{2,\sigma})(2\sigma t,x) \d x \d t.
\eeas
Next, making the substitution $s=\sigma t$ in the above integral, we get that
\beas
& & \sigma \abs{\int_0^T \int_{\R^3} \theta \cdot \left( \tilde{A}^\sharp(x)-A_\sigma^\sharp(x) \right) (\overline{\phi_1} \phi_2)(x'-2\sigma t\theta,x_3) (\overline{b_{1,\sigma}} b_{2,\sigma})(2\sigma t,x) \d x \d t} \\
& = & \abs{\int_0^{+\infty}\int_{B(0,R) \times \R} \theta \cdot \left( \tilde{A}^\sharp(x)-A_\sigma^\sharp(x) \right) (\overline{\phi_1} \phi_2)(x'-2 s \theta,x_3) (\overline{b_{1,\sigma}} b_{2,\sigma})(2 s,x) \d x \d s} \\
& \leq & \norm{\tilde{A}^\sharp-A_\sigma^\sharp}_{L^\infty (\R^3)^2} \int_0^{R+1}\int_{B(0,R) \times \R} \abs{(\overline{\phi}_1 \phi_2)(x'-2s\theta,x_3)} \d x \d s \\
& \leq & \norm{\tilde{A}^\sharp-A_\sigma^\sharp}_{L^\infty (\R^3)^2} \int_0^{R+1}\int_{\R^3} \abs{(\overline{\phi_1} \phi_2)(x'-2s\theta,x_3)} \d x \d s \\
& \leq & (R+1) \norm{\tilde{A}^\sharp-A_\sigma^\sharp}_{L^\infty (\R^3)^2} \| \phi_1 \|_{L^2(\R^3)} \| \phi_2 \|_{L^2(\R^3)}.
\eeas
From this and \eqref{mol1} with $k=0$, and \eqref{def-nrm}, it follows that
\bea
& & \sigma \abs{\int_0^T \int_{\R^3} \theta \cdot \left( \tilde{A}^\sharp(x)-A_\sigma^\sharp(x) \right) (\overline{\phi_1} \phi_2)(x'-2\sigma t\theta,x_3) (\overline{b_{1,\sigma}} b_{2,\sigma})(2\sigma t,x) \d x \d t} \nonumber \\
& \leq & C \sigma^{-2 \slash 3} \norm{\phi_1}_{L^2(\R^3)} \norm{\phi_2}_{L^2(\R^3)} \leq C \sigma^{-4/3} \cN_{\theta,\sigma}(\phi_1) \cN_{\theta,\sigma}(\phi_2). \label{66.1} 
\eea
On the other hand, since
\beas
(\overline{b_{1,\sigma}} b_{2,\sigma})(2\sigma t,x'+2 \sigma t \theta,x_3) & = & \exp\para{-i \int_0^{2\sigma t} \theta \cdot A_\sigma^\sharp(x'+(2 \sigma t -s)\theta,x_3) \d s} \\
& =& \exp\para{-i \int_0^{2\sigma t} \theta \cdot A_\sigma^\sharp(x'+s\theta,x_3) \d s},
\eeas 
for a.e. $(t,x) \in (0,T) \times \R^3$, we have
\bea
& & \sigma \int_0^T \int_{\R^3} \theta\cdot A_\sigma^\sharp(x) (\overline{\phi_1} \phi_2)(x'- 2 \sigma t \theta,x_3)
(\overline{b_{1,\sigma}} b_{2,\sigma})(2\sigma t,x) \d x' \d x_3 \d t \nonumber \\
& = & \sigma \int_0^T \int_{\R^3} \theta \cdot A_\sigma^\sharp(x'+2 \sigma t \theta,x_3) (\overline{\phi_1} \phi_2)(x) (\overline{b_{1,\sigma}} b_{2,\sigma})(2\sigma t, x' + 2 \sigma
t \theta,x_3) \d x' \d x_3 \d t \nonumber \\
& = & \int_{\R^3} (\overline{\phi}_1 \phi_2)(x) \left( \int_0^T  \sigma \theta \cdot A_\sigma^\sharp(x'+2\sigma t\theta,x_3) \exp\para{-i\int_0^{2\sigma t} \theta \cdot
A_\sigma^\sharp(x'+s\theta,x_3) \d s} \d t \right) \d x' \d x_3 \nonumber \\
& = & \frac{i}{2} \int_{\R^3} (\overline{\phi}_1 \phi_2)(x) \left( \int_0^T \frac{\d}{\d t} \exp \para{-i \int_0^{2\sigma t} \theta \cdot
A_\sigma^\sharp(x'+s\theta,x_3) \d s} \d t \right) \d x' \d x_3 \nonumber \\ 
& = & \frac{i}{2}\int_{\R^3} (\overline{\phi}_1 \phi_2)(x) \left( \exp\para{-i\int_0^{2\sigma
T}\theta\cdot A^\sharp_\sigma(x'+s\theta,x_3) \d s}-1 \right) \d x' \d x_3. \label{6.2}
\eea
As $A_\sigma^\sharp$ is supported in $B(0,R) \times \R$ and $|x' + s  \theta| > 2 \sigma_* T - (R+1) > R$, for all $x' \in \cD_R$ and all $s > 2 \sigma T$, we have
\bel{a22}
\int_0^{2\sigma T} \theta \cdot A_\sigma^\sharp (x'+s\theta,x_3) \d s =\int_0^{+\infty} \theta \cdot A_\sigma^\sharp(x'+s\theta,x_3) \d s,\ x' \in \cD_R,\ x_3 \in \R.
\ee
Similarly, as $\abs{x'+s\theta}^2=\abs{x'}^2+s^2+2sx'\cdot \theta > R^2$, for every $x'\in \cD_R^-(\theta)$ and $s < 0$, we have
$$ \int_{-\infty}^0 \theta \cdot A_\sigma^\sharp(x'+s\theta,x_3) \d s = 0,\ x' \in \cD_R^-(\theta),\ x_3 \in \R. $$
This and \eqref{a22} entail
\bel{6.3}
\int_0^{2\sigma T} \theta \cdot A_\sigma^\sharp(x'+s\theta,x_3) \d s = \int_\R \theta \cdot A_\sigma^\sharp(x'+s\theta,x_3) \d s,\ x' \in \cD_R^-(\theta),\ x_3 \in \R.
\ee
Having seen this, we take $\phi_1:=\pd_{x_j} \overline{\phi}$, for $j=1,2,3$, and $\phi_2:=\phi$, in \eqref{6.2}, and find
\bea
& & \sigma \int_0^T \int_{\R^3} \theta \cdot A_\sigma^\sharp(x) (\overline{\phi_1} \phi_2)(x'-2\sigma t\theta,x_3)
(\overline{b_{1,\sigma}} b_{2,\sigma})(2\sigma t,x) \d x' \d x_3 \d t \nonumber \\
& = & \frac{i}{4}\int_{\R^3} \pd_{x_j} \phi^2(x)  \left( \exp\para{-i\int_0^{2\sigma
T}\theta\cdot A^\sharp_\sigma(x'+s\theta,x_3) \d s}-1 \right) \d x' \d x_3 \nonumber \\
& = & -\frac{1}{4} \int_{\R^3} \phi^2(x) \left( \int_0^{2\sigma T} \theta \cdot
\pd_{x_j} A_\sigma^\sharp(x'+s\theta,x_3) \d s \right) \exp \para{-i \int_0^{2\sigma T} \theta \cdot
A_\sigma^\sharp(x'+s\theta,x_3) \d s} \d x, \label{a21}
\eea
upon integrating by parts. Taking into account that $\phi$ is supported in $\cD_R^-(\theta) \times \R$, we deduce from \eqref{6.3} and \eqref{a21}, that
\bea
& & \sigma \int_0^T \int_{\R^3} \theta \cdot A_\sigma^\sharp(x) (\overline{\phi_1} \phi_2)(x'-2\sigma t\theta,x_3)
(\overline{b_{1,\sigma}} b_{2,\sigma})(2 \sigma t,x) \d x' \d x_3 \d t \nonumber \\
& =& -\frac{1}{4} \int_{\R^3} \phi^2(x) \para{\int_\R \theta \cdot \pd_{x_j} A_\sigma^\sharp(x'+s\theta,x_3) \d s} \exp\para{-i \int_\R \theta \cdot A_\sigma^\sharp(x'+s\theta,x_3) \d s} \d x' \d x_3 \nonumber \\
& = & -\frac{1}{4} \int_{\R^3} \phi^2(x) \cP(\rho_{j,\sigma})(\theta,x',x_3)
\exp\para{-i \int_\R \theta \cdot A_\sigma^\sharp(x'+s\theta,x_3) \d s} \d x' \d x_3. \label{6.4}
\eea
Here we used \eqref{def-P} and the notation
$$
\rho_{j,\sigma}(x):= \theta \cdot \pd_{x_j} A_\sigma^\sharp(x) = \sum_{i=1,2} \theta_i \pd_{x_j} a_{i,\sigma}(x),\ x \in\R^3,\ j=1,2,3.
$$
Finally, using once more that the functions $A_\sigma^\sharp$ and $\tilde{A}^{\sharp}$ are supported in $B(0,R)$, we infer from \eqref{mmol1} and \eqref{6.1}-\eqref{def-P}, that
$$ \abs{\left( \cP(\rho_{j,\sigma})-\cP(\tilde{\rho}_j) \right)(\theta,x',x_3)} \leq C \sigma^{-1/3},\ (x',x_3) \in B(0,R) \times \R, $$
for some positive constant $C$, depending only on $\omega$ and $M$. This entails that
$$ \abs{\int_{\R^3} \phi^2(x) \left( \cP(\rho_{j,\sigma})-\cP(\tilde{\rho}_j) \right)(\theta,x',x_3)
\exp \para{-i \int_\R \theta \cdot A_\sigma^\sharp(x'+s\theta,x_3) \d s} \d x' \d x_3} \leq C \sigma^{-1/3} \norm{\phi}_{L^2(\R^3)}^2, $$
which, 
together with \eqref{5.11},\eqref{66.1}, and \eqref{6.4}, yields \eqref{66.2}.
\end{proof}

As will be seen in the coming section, the result of Lemma \ref{L6.2} is a key ingredient in the estimation of the partial Fourier transform of the aligned magnetic field, in terms of the DN map. To this purpose, we recall for all $f \in \sX$, where $\sX$ is defined in \eqref{def-X}, that the partial Fourier transform with respect to $x' \in \R^2$ of $f$, expresses as
\bel{def-FT}
\widehat{f}(\xi',x_3) :=(2\pi)^{-1}\int_{\R^2} f(x',x_3) e^{-i x'\cdot \xi'} \d x',\ \xi' \in \R^2,\ x_3 \in \R.
\ee
Further, setting $\theta^\perp:= \{ x' \in \R^2;\ x ' \cdot \theta =0 \}$, we recall for further use from \cite[Lemma 6.1]{[BC2]}, that $x' \mapsto \cP(f)(\theta,x',x_3) \in L^1(\theta^\perp)$ for a.e. $x_3 \in \R$, and that
\bel{a30}
\widehat{\cP(f)}(\theta,\xi',x_3) := (2\pi)^{-{1 \slash 2}} \int_{\theta^\perp} \cP(f)(\theta,x',x_3) e^{-ix'\cdot \xi'} \d x' = (2\pi)^{1 \slash 2} \widehat{f}(\xi',x_3),\ \xi'\in \theta^\perp,\ x_3 \in \R.
\ee

\subsection{Aligned magnetic field estimation}

Let us now estimate the Fourier transform of the aligned magnetic field
\bel{def-beta}
\tilde{\beta}(x) :=\left( \pd_{x_1} \tilde{a}_2 - \pd_{x_2} \tilde{a}_1 \right)(x),\ x \in \R^3,
\ee
with the aid of Lemma \ref{L6.2}. More precisely, we aim to establish the following result.
\begin{lemma}
\label{L7.1}
Let $M>0$, and let $A_j$ and $q_j$, for $j=1,2$, be as in Proposition \ref{pr-4.1}.
Then, there exist two constants $\epsilon \in (0,1)$ and $C>0$, both of them depending only on $T$, $\omega$, and $M$, such that the estimates
\bel{7.5}
\norm{\widehat{\beta}(\xi',\cdot)}_{L^\infty(\R)}\leq C \langle \xi' \rangle^7 \para{\sigma^6 \norm{\Lambda_{A_1,q_1}-\Lambda_{A_2,q_2}} + \sigma^{-\epsilon}},
\ee
and
\bel{7.6aa}
\norm{\pd_{x_3} \widehat{\beta}(\xi',\cdot)}_{L^\infty(\R)}\leq C \langle \xi' \rangle^8 \para{\sigma^6\norm{\Lambda_{A_1,q_1}-\Lambda_{A_2,q_2}} + \sigma^{-\epsilon}},
\ee
hold for all $\sigma >\sigma_*$ and all $\xi'\in\R^2$, with $\seq{\xi'}:=\left( 1 + |\xi'|^2 \right)^{1 \slash 2}$.
\end{lemma}

\begin{proof}
We shall only prove \eqref{7.5}, the derivation of \eqref{7.6aa} being obtained in a similar fashion.

Fix $\theta \in \mathbb S^1 \cap \xi'^\perp$. We first introduce the following partition of $B(0,R) \cap \theta^\perp$. For $N \in \mathbb{N}^*:=\{1,2,\ldots \}$ fixed, we pick $x_1',\ldots,x_N'$ in $B(0,R+1 \slash 2) \cap \theta^\perp$, and choose $\varphi_1,\ldots,\varphi_N$ in $\cC_0^\infty(\R^2,[0,1])$, such that
\bel{partition}
\supp\ \varphi_k \subset  B(x_k', 1 \slash 8) \cap \theta^\perp\ \mbox{for}\ k=1,\ldots,N,\ \mbox{and}\
\sum_{k=1}^N \varphi_k(x')=1\ \mbox{for}\ x' \in B(0,R) \cap \theta^\perp.
\ee
Next, we set $r_{x_k'}:=\left( \para{R+3 \slash 4}^2-\abs{x_k'}^2 \right)^{1 \slash 2}$, in such a way that
\bel{a23}
B(x_k'- r_{x_k'} \theta,1 \slash 4) \subset \cD^-_R(\theta),\ k=1,\ldots,N.
\ee
In order to define a suitable set of test functions $\phi_{*,k}$, $k=1,\ldots,N$, we fix $x_3 \in \R$, pick a function $\alpha \in \cC_0^\infty(\R,\R_+)$ which is supported in $(-1,1)$ and normalized in $L^2(\R)$, and put
\bel{def-alpha}
\alpha_{\sigma}(s) := \sigma^\mu \alpha \left( \sigma^{2\mu} (x_3-s) \right),\ s \in \R,
\ee
for some positive real parameter $\mu$, we shall make precise below. Then, the test function $\phi_{*,k}$ is defined for all $y= (y',y_3) \in \R^3$, by
\bel{7.1}
\phi_{*,k}(y) := h \left( y' \cdot \theta + r_{x'_k} \right) e^{-\frac{i}{2} y' \cdot \xi'} \varphi_k^{1/2}(y'-(y' \cdot \theta)\theta)
\exp \para{\frac{i}{2} \int_\R \theta \cdot A_\sigma^\sharp(y'+s \theta , y_3) \d s} \alpha_{\sigma}(y_3),
\ee
where $h \in \cC_0^\infty(\R)$ is supported in $(0,1/8)$, and normalized in $L^2(\R)$.

For every $y' \in \R^2 \setminus B(x_k'-r_{x_k'}\theta,1 \slash 4)$, it is easily seen from the basic inequality 
$$|y'-(x_k'-r_{x_k'} \theta)| \leq |y'-(y'\cdot\theta)\theta-x_k'|+|y' \cdot \theta+r_{x_k'}|, $$ 
that either of the two real numbers $|y'- (y' \cdot \theta) \theta-x_k'|$ or $|y' \cdot \theta+r_{x_k'}|$ is greater than $1 \slash 8$, and hence that
$h(y' \cdot \theta + r_{x_k'}) \varphi_k^{1/2}(y'- (y'\cdot \theta) \theta)=0$. As a consequence, we have
\bel{a25}
\supp\ \phi_{*,k}(\cdot,y_3) \subset B \left( x_k'-r_{x_k'} \theta , 1 \slash 4 \right) \subset \cD^-_R(\theta),\ y_3 \in \R,\ k=1,\ldots,N,
\ee
directly from \eqref{a23} and \eqref{7.1}.
Moreover, since
$$
\theta \cdot \nabla_{y'} \para{\int_\R \theta \cdot A_\sigma^\sharp(y'+s\theta,y_3) \d s} 
= \theta  \cdot \int_{\R} \frac{\d}{\d s} A_\sigma^\sharp(y'+s\theta,y_3) \d s =0,\ (y',y_3) \in \R^3,
$$
we derive from Lemma \ref{Lmol} for all $m \in \mathbb{N}$, that 
$$
\seq{\xi'} \norm{\phi_{*,k}}_{H^m(\R^3)}  + 
\norm{\theta \cdot \nabla_{x'} \phi_{*,k}}_{H^m(\R^3)} \leq C \seq{\xi'}^{m+1} \sigma^{2\mu m + \max \left( 0,(m-2) \slash 3 \right)},
$$
where $C$ is a positive constant, independent of $\sigma$. Therefore, we have $N_{0,\theta}(\phi_{*,k}) \leq C \seq{\xi'}$ and $N_{2,\theta}(\phi_{*,k}) \leq C \seq{\xi'}^3 \sigma^{4 \mu}$, whence
\bel{a23b}
\cN_{\theta,\sigma}(\phi_{*,k}) \leq C \seq{\xi'}^3 \sigma^{4 \mu + 1 \slash 3}.
\ee
Similarly, we find that
$$
\seq{\xi'} \norm{\pd_{x_j} \phi_{*,k}}_{H^m(\R^3)} + 
\norm{\theta \cdot \nabla_{x'} \pd_{x_j} \phi_{*,k}}_{H^m(\R^3)} \leq C \seq{\xi'}^{m+2} \sigma^{2 \mu m + \max \left( 0,(m-1) \slash 3 \right)},\ j=1,2,
$$
and
$$
\seq{\xi'} \norm{\pd_{x_3} \phi_{*,k}}_{H^m(\R^3)} +  
\norm{\theta \cdot \nabla_{x'} \pd_{x_3} \phi_{*,k}}_{H^m(\R^3)} \leq C \seq{\xi'}^{m+1} \sigma^{2 \mu (m+1) + \max \left( 0,(m-1) \slash 3 \right)}.
$$
Thus, we have $N_{0,\theta}(\pd_{x_j} \phi_{*,k}) \leq C \seq{\xi'}^2 \sigma^{2 \mu}$ and $N_{2,\theta}(\pd_{x_j} \phi_{*,k}) \leq C \seq{\xi'}^4 \sigma^{6 \mu+ 1 \slash 3}$, for $j=1,2,3$, and consequently
$$
\cN_{\theta,\sigma}(\pd_{x_j} \phi_{*,k}) \leq C \seq{\xi'}^4 \sigma^{6 \mu + 1 \slash 3},\ j=1,2,3.
$$
according to \eqref{def-nrm}. From this and \eqref{a23b}, it then follows that
\bel{a24}
\cN_{\theta,\sigma}(\phi_{*,k}) \cN_{\theta,\sigma}(\pd_{x_j} \phi_{*,k}) \leq  C \seq{\xi'}^7 \sigma^{10 \mu + 2 \slash 3},\ j=1,2,3.
\ee

Having seen this, we turn now to estimating $\widehat{\tilde{\rho}_j}$, where $\tilde{\rho}_j$ is defined by \eqref{6.1}. 
As $A_\sigma^\sharp\in W^{\infty,\infty}(\R^3,\R)^2$, we infer form \eqref{7.1} that $\phi_{*,k} \in \cC_0^\infty(\R^3)$, and from \eqref{a25} that $\supp\ \phi_{*,k} \subset \cD_R \times \R$.
Thus, by performing the change of variable $y'=x'+t \theta \in \theta^\perp \oplus \R\theta$, in the following integral, we deduce from \eqref{def-alpha}-\eqref{7.1} that
\bea
& & \int_\R \int_{\R^2} \phi_{*,k}^2(y',y_3) \cP(\tilde{\rho}_j)(\theta,y',y_3)
\exp \para{-i \int_\R \theta \cdot A_\sigma^\sharp(y'+s\theta,y_3) \d s} \d y' \d y_3 \nonumber \\
& = & \int_{\R} \int_{\R} \int_{\theta^\perp} \phi_{*,k}^2(x'+t\theta,y_3) \cP(\tilde{\rho}_j)(\theta,x'+t \theta,y_3)
\exp \para{- i \int_\R \theta \cdot A_\sigma^\sharp(x'+s\theta,y_3) \d s} \d x' \d t \d y_3 \nonumber \\ 
& = & \int_{\R} \int_{\R} \int_{\theta^\perp} h^2(t+r_{x_k'})e^{-i x' \cdot \xi'} \varphi_k(x') \alpha_{\sigma}^2(y_3) \cP(\tilde{\rho}_j)(\theta,x',y_3) \d x' \d t \d y_3 \nonumber \\
& = & \int_\R \int_{\theta^\perp} e^{-i y'\cdot\xi'} \varphi_k(y') \alpha_{\sigma}^2(y_3) \cP(\tilde{\rho}_j)(\theta,y',y_3) \d y' \d y_3. \label{7.3}
\eea
Thus, taking $\mu>0$ so small that $\kappa:=1/6-10\mu>0$, we deduce from this, \eqref{66.2}, and \eqref{a24}, that
\bea
& & \abs{\int_\R \int_{\theta^\perp} e^{-i y' \cdot \xi'} \varphi_k(y') \alpha_{\sigma}^2(y_3) \cP(\tilde{\rho}_j)(\theta,y',y_3) \d y' \d y_3} \nonumber \\
& \leq & C \langle \xi' \rangle^7 \para{\sigma^6 \norm{\Lambda_{A_1,q_1}-\Lambda_{A_2,q_2}} + \sigma^{-\kappa}},\ x_3 \in \R. \label{dd}
\eea
Moreover, we see from \eqref{6.1} that $\tilde{\rho}_j \in \cC^{0,1}(\R^3)$. Since $\supp\ \tilde{\rho}_j \subset B(0,R) \times \R$, by Lemma \ref{ll1m}, then $x \mapsto \cP(\tilde{\rho}_j)(\theta,x) \in \cC^{0,1}(\R^3)$, and we deduce from \eqref{mol2} upon making the substitution $s=\sigma^{2\mu}(x_3-y_3)$ in the following integral, that
\beas
& & \abs{\int_\R \int_{\theta^\perp} e^{-i y'\cdot\xi'} \varphi_k(y') \cP(\tilde{\rho}_j)(\theta,y',y_3) \alpha_{\sigma}^2(y_3) \d y' \d y_3 - \int_{\theta^\perp} e^{-i y' \cdot \xi'} \varphi_k(y') \cP(\tilde{\rho}_j)(\theta,y',x_3) \d y'} \\
& = & \abs{\int_\R \int_{\theta^\perp} e^{-iy' \cdot \xi'} \varphi_k(y') \alpha^2(s) \left( \cP(\tilde{\rho}_j)(\theta,y',x_3-\sigma^{-2\mu}s)-\cP(\tilde{\rho}_j)(\theta,y',x_3)\right) \d y' \d s} \\
& \leq & \int_{-1}^1 \int_{\theta^\perp \cap B(0,R+1)} \alpha^2(s) \abs{\cP(\tilde{\rho}_j)(\theta,y',x_3-\sigma^{-2\mu}s)-\cP(\tilde{\rho}_j)(\theta,y',x_3)} \d y' \d s \leq C \sigma^{-2\mu},
\eeas
for some constant $C>0$ depending only on $\omega$ and $M$. Here, we used the fact that $\phi_{*,k}$ and $\alpha$ are supported in $B(0,R+1)$ and $(-1,1)$, respectively. This and \eqref{dd} yield 
\bel{a26}
\abs{\int_{\theta^\perp} e^{-i y'\cdot \xi'} \varphi_k(y') \cP(\tilde{\rho}_j)(\theta,y',x_3) \d y'}
\leq
C \langle \xi' \rangle^7\para{\sigma^6 \norm{\Lambda_{A_1,q_1}-\Lambda_{A_2,q_2}}+ \sigma^{-2\mu}+\sigma^{-\kappa}},
\ee
for all $x_3 \in \R$ and $k=1,...,N$.
Further, as $A^\sharp$ is supported in $B(0,R) \times \R$ by assumption, it holds true that $\pd_{x_j} A^\sharp(y'+ s \theta,x_3) = 0$ for all $s \in \R$, $x_3 \in \R$, and all $y' \in \theta^\perp$ such that $|y'| \geq R$. Therefore, we have
$$ \cP(\tilde{\rho}_j)(\theta,y',x_3)=0,\ y' \in \theta^\perp \cap (\R^2 \setminus B(0,R)),\ x_3 \in \R, $$
in virtue of \eqref{6.1}, and hence
$$ \int_{\theta^\perp} e^{-i y'\cdot \xi'} \cP(\tilde{\rho}_j)(\theta,y',x_3) \d y' 
= \int_{\theta^\perp \cap B(0,R)} e^{-i y'\cdot \xi'} \cP(\tilde{\rho}_j)(\theta,y',x_3) \d y',\ x_3 \in \R.
$$
In light of \eqref{a30} and \eqref{partition}, this entails that
\bel{a27}
\widehat{\tilde{\rho}}_j(\xi',x_3) = \frac{1}{2 \pi} \sum_{k=1}^N \int_{\theta^\perp \cap B(0,R)} e^{-i y'\cdot \xi'} \varphi_k(y') \cP(\tilde{\rho}_j)(\theta,y',x_3) \d y',\ x_3 \in \R. 
\ee
Taking $\mu \in (0,1 \slash 72]$, in such a way that we have $\kappa \geq 2 \mu$, we infer from \eqref{a26}-\eqref{a27} that
\bea
\norm{\widehat{\tilde{\rho}}_j(\xi',\cdot)}_{L^\infty(\R)} & \leq & \sum_{k=1}^N \left( \sup_{x_3 \in \R} \abs{\int_{\theta^\bot}e^{-i y'\cdot\xi'}\varphi_k(y') \cP(\tilde{\rho}_j)(\theta,y',x_3)dy'} \right) \nonumber \\
& \leq &  C \langle \xi' \rangle^7 \para{\sigma^6 \norm{\Lambda_{A_1,q_1}-\Lambda_{A_2,q_2}} + \sigma^{-2\mu}},\ x_3 \in \R. \label{7.4}
\eea
The last step of the proof is to notice from \eqref{6.1}, \eqref{def-FT}, and the identity $\sum_{m=1,2} \theta_m \xi_m = \theta \cdot \xi'= 0$, that
$$
\widehat{\tilde{\rho}}_j(\xi',x_3)=i \sum_{m=1,2} \theta_m \xi_j \widehat{\tilde{a}}_m(\xi',x_3)
= i \sum_{m=1,2} \theta_m \para{\xi_j \widehat{\tilde{a}}_m -\xi_m
\widehat{\tilde{a}}_j}(\xi',x_3),\ x_3 \in \R,\ j=1,2.
$$
Thus, assuming that $\xi'=(\xi_1,\xi_2) \in \R^2 \setminus \{ 0 \}$, we get from \eqref{def-beta} upon 
choosing $\theta=(\xi_2 \slash \abs{\xi'},-\xi_1 \slash \abs{\xi'})$, that
$$
\widehat{\tilde{\rho}}_j(\xi',x_3)=-\frac{\xi_j}{\abs{\xi'}}\widehat{\tilde{\beta}}(\xi',x_3),\ x_3 \in \R.
$$
From this and \eqref{7.4}, it then follows that
$$
\norm{\widehat{\tilde{\beta}}(\xi',\cdot)}_{L^\infty(\R)} \leq \frac{\abs{\xi_1}+\abs{\xi_2}}{\abs{\xi'}} \norm{\widehat{\tilde{\beta}}(\xi',\cdot)}_{L^\infty(\R)} \leq C \langle \xi' \rangle^7 \para{\sigma^6 \norm{\Lambda_{A_1,q_1}-\Lambda_{A_2,q_2}} + \sigma^{-2\mu}},
$$
which yields \eqref{7.5} for $\xi' \neq 0$. Since $\widehat{\tilde{\beta}}(0,x_3)=0$ for every $x_3 \in \R$, from \eqref{def-beta}, then \eqref{7.5} holds for $\xi'=0$ as well, and the proof is complete.
\end{proof}

Armed with Lemma \ref{L7.1}, we turn now to proving the three main results of this paper.

\section{Proof of Theorems \ref{Th1}, \ref{Th2} and \ref{Th22}}
\setcounter{equation}{0}
Let us start by reducing the analysis of the inverse problem under investigation to the case of transverse magnetic potentials.
To do that, we consider $A'=(a_i')_{1 \leq i \leq 3} \in \cA$, and put $A:=(a_{1},a_2,0)$, where
\bel{p1}
a_i(x',x_3):= a_i'(x',x_3)- \int_{-\infty}^{x_3} \pd_{x_i} a_3' (x',s) \d s,\ x=(x',x_3) \in \omega \times \R,\ i=1,2.
\ee
Since $a_3' \in C^3(\overline{\Omega})$ fulfills \eqref{decay}-\eqref{bord}, from the very definition of $\cA$, we have
$a_3' \in L_{x_3}^1(\R, H_0^3(\omega))$, where $H_0^3(\omega)$ denotes the closure of $\cC^\infty_0(\omega)$ in $H^3(\omega)$.
Thus, $e(x) := \int_{-\infty}^{x_3} a_{3}'(x',s) \d s$ lies in $W^{3,\infty}(\Omega)\cap L_{x_3}^\infty(\R, H_0^3(\omega))$, and we deduce from the identity 
$A=A'-\nabla e$, arising from \eqref{p1}, that 
$$
\d A' = \d A,\ \mbox{and}\ \Lambda_{A_*+A',q}=\Lambda_{A_*+A,q},
$$
for all $A_*\in W^{2,\infty}(\Omega)^3$ and all $q \in W^{1,\infty}(\Omega)$. 
Moreover, it is easy to see that $A$ obeys \eqref{decay}, in the sense that we have
\bel{cond1-b}
\pd_x^\alpha A(x)=0,\ x \in \pd \Omega,\ \alpha \in \mathbb N^3,\ |\alpha | \leq 1.
\ee
Therefore, for each $A_* \in W^{2,\infty}(\Omega,\R)^3$ and any $A_j \in A_*+\cA$, $j=1,2$, we may assume without loss of generality, that the difference $A_2-A_1$ reads 
\bel{p5}
A=(a_1,a_2,0),
\ee
and fulfills \eqref{cond1-b}. We shall systematically do that in the sequel. For further reference, we put $A^\sharp:=(a_1,a_2)$, where $a_j$, $j=1,2$, are extended by zero outside $\Omega$. Summing up, we have

\subsection{Proof of Theorem \ref{Th1}}
We establish the uniqueness result $(\d A_1,q_1)=(\d A_2, q_2)$ in Subsection \ref{sec-uniqueness-th1}, while
the proof of the stability estimate \eqref{t3bb} can be found in Subsection \ref{sec-stability-th1}.

\subsubsection{Uniqueness result}
\label{sec-uniqueness-th1}
For $\xi'=(\xi_1,\xi_2)\in \R^2 \setminus \{0\}$, we set $\xi'_\perp:=(- |\xi'|^{-1} \xi_2 , |\xi'|^{-1} \xi_1)$, and we decompose $A^\sharp$ into the sum $(A^\sharp \cdot \xi') |\xi'|^{-2} \xi'+ (A^\sharp \cdot \xi_\perp') \xi_\perp'$, in such a way that the partial Fourier transform 
of $\pd_{x_3} A^\sharp$, reads
\bel{p2}
\pd_{x_3} \widehat{A^\sharp}(\xi',x_3) = \left(\frac{1}{2\pi} \int_{\R^2} e^{-ix'\cdot\xi'} \pd_{x_3} A^\sharp(x',x_3) \cdot \xi' \d x' \right) \frac{\xi'}{|\xi'|^2} +
i \frac{\pd_{x_3} \widehat{\beta}(\xi',x_3)}{|\xi'|} \xi'_\perp,\ x_3 \in \R.
\ee
Next, recalling the hypothesis $\Lambda_{A_1,q_1}=\Lambda_{A_2,q_2}$, we get
\bel{t2b}
\beta = \pd_{x_1} a_2-\pd_{x_2} a_1=0\ \mbox{in}\ \Omega,
\ee
upon sending $\sigma$ to infinity in \eqref{7.5}. Moreover, we have $\nabla_{x'} \cdot \pd_{x_3} A^\sharp=\nabla \cdot \pd_{x_3} A=0$, in virtue of \eqref{t3a}, whence
\bea
\int_{\R^2} e^{-i x' \cdot \xi'} \pd_{x_3} A^\sharp(x',x_3) \cdot \xi' \d x' & = &  i \int_{\R^2} \nabla_{x'} e^{-i x' \cdot \xi'} \cdot \pd_{x_3} A^\sharp(x',x_3) \nonumber \\
& = & - i \int_{\R^2} e^{-ix'\cdot\xi'} \nabla_{x'} \cdot \pd_{x_3} A^\sharp(x',x_3) \d x'=0. \label{p2b}
\eea
Putting this together with \eqref{p2}-\eqref{t2b}, we find that $| \xi' | \pd_{x_3} \widehat{A^\sharp}(\xi',x_3)=0$ for a.e. $x_3 \in \R$. Since $\xi'$ is arbitrary in $\R^2 \setminus \{0\}$, this entails that $\pd_{x_3} A^\sharp=0$, and hence that $\pd_{x_3} a_1=\pd_{x_3} a_2=0$ in $\R^2$. From this, \eqref{t2b}, and the fact that $a_3$ is uniformly zero, it then follows that $\d A_1= \d A_2$. 

Further, taking into account that $\pd_x^\alpha A_1=\pd_x^\alpha A_2=\pd_x^\alpha A_*$ on $\pd \Omega$, for every $\alpha \in \mathbb N^3$ such that $|\alpha|\leq 1$,
we infer that $A \in W^{2,\infty}(\R^3,\R)^3$. This and the identity $\d A=0$, yield $A=\nabla \Psi$, where the function
$\Psi(x)=\int_0^1 x \cdot A(t x) \d t$ lies in $W^{3,\infty}(\R^3,\R)$. Moreover, since $A$ vanishes in $\R^3 \setminus \Omega$, we may assume upon possibly adding a suitable constant, that the same is true for $\Psi$. Therefore, $\Psi_{|\pd \Omega}=0$, and we find
$\Lambda_{A_2,q_2}=\Lambda_{A_2+ \nabla \Psi,q_2}=\Lambda_{A_1,q_2}$,
by combining the identity $A_1=A_2+\nabla\Psi$ with the gauge invariance property of the DN map.
From this and the assumption $\Lambda_{A_1,q_1}=\Lambda_{A_2,q_2}$, it then follows that
\bel{p3}
\Lambda_{A_1,q_2}=\Lambda_{A_1,q_1}.
\ee
It remains to show that the function $q=q_2-q_1$, duly extended by zero outside $\Omega$, is uniformly zero in $\R^3$.
This can be done upon applying the orthogonality identity \eqref{5.5} with $A_1=A_2$, i.e with $A=0$ and $V=q$. In light of \eqref{p3}, we obtain that
\bel{p4}
\langle q u_{2,\sigma} , u_{1,\sigma} \rangle_{L^2(Q)} = 0,\ \sigma>\sigma_*.
\ee
Here $u_{j,\sigma}$, for $j=1,2$, is given by \eqref{GO}, and we have $(\overline{b_{1,\sigma}}b_{2,\sigma})(t,x)=1$ for all $(t,x) \in (0,T) \times \R^3$, from \eqref{def-b}, since $A_1=A_2$. 

Next, pick $\phi \in \cC_0^{\infty}(\R^3)$, with support in $\{ x \in \R^3; |x| < 1 \}$, and such that $\| \phi \|_{L^2(\R^3)}^2=1$. We fix $y \in \cD_R(\theta) \times \R$, and choose $\delta >0$ so small that $\phi_1(x)=\phi_2(x):=\delta^{-3 \slash 2} \phi(\delta^{-1}(x-y))$ is supported in $\cD_R \times \R$. Thus, upon multiplying \eqref{p4} by $\sigma$, and then sending $\sigma$ to infinity, we find with the aid of \eqref{a1b} and \eqref{4.6}, that
\bel{p4b}
\int_0^{+\infty} \left( \int_{\R^3} q(\delta x'+ y' + s \theta, \delta x_3 + y_3) | \phi(x',x_3) |^2 \d x' \d x_3 \right) \d s = 0,\ \delta >0.
\ee
Actually, if $y' \in \cD_R^-(\theta)$, then we have $|y'+ s \theta | > R$ for any $s \leq 0$, and hence $q(\delta x'+ y' + s \theta, \delta x_3 + y_3)=0$, uniformly in $|x|<1$, provided $\delta \in (0,1)$. This and \eqref{p4b} yield that
\bel{p4c} 
\int_{\R} \left( \int_{\R^3} q(\delta x'+ y' + s \theta, \delta x_3 + y_3) | \phi(x',x_3) |^2 \d x' \d x_3 \right) \d s = 0,\ \delta \in (0,1),\ (y',y_3) \in \cD_R^-(\theta) \times \R.
\ee
By performing the change of variable $t=-s$ in the above integral, and then substituting $(-\theta)$ for $\theta$ in the resulting identity, we get that
$$ \int_{\R} \left( \int_{\R^3} q(\delta x'+ y' +s \theta, \delta x_3 + y_3) | \phi(x',x_3) |^2 \d x' \d x_3 \right) \d s = 0,\ \delta \in (0,1),\ (y',y_3) \in \cD_R^-(-\theta) \times \R.$$
This and \eqref{p4c} yield that
$$
\int_{\R} \left( \int_{\R^3} q(\delta x'+ y' + s \theta, \delta x_3 + y_3) | \phi(x',x_3) |^2 \d x' \d x_3 \right) \d s = 0,\ \delta \in (0,1),\ (y',y_3) \in \cD_R \times \R.
$$
Next, sending $\delta$ to zero in the above identity, and taking into account that $\phi$ is normalized in $L^2(\R^3)$, we obtain for each $\theta \in \mathbb S^1$, that
$$ \cP(q)(\theta,y',y_3)= \int_\R q(y'+s \theta,y_3) \d s =0,\ (y',y_3) \in \D_R \times \R. $$
This entails $q=0$, since the partial X-ray transform is injective.

\subsubsection{Proof of the stability estimate \eqref{t3bb}}
\label{sec-stability-th1}
	
We have 
$$ |\xi'| \abs{\pd_{x_3} \widehat{A^\sharp}(\xi',x_3)}=\abs{\pd_{x_3} \widehat{\beta}(\xi',x_3)},\ \xi' \in \R^2,\ x_3 \in \R, $$
by \eqref{p2} and \eqref{p2b}, so we infer from \eqref{7.5}-\eqref{7.6aa} that
\bel{t3d}
| \widehat{\beta}(\xi',x_3) | + |\xi'| | \pd_{x_3} \widehat{A^\sharp}(\xi',x_3) | \leq 
C \langle \xi' \rangle^8 \para{\sigma^6 \norm{\Lambda_{A_1,q_1}-\Lambda_{A_2,q_2}}+\sigma^{-\epsilon}},\ \xi'\in\R^2,\ \ x_3 \in\R,
\ee
for all $\sigma>\sigma_*$, the constants $C$ and $\epsilon$ being the same as in Lemma \ref{L7.1}.

Fix $\rho \in (1,+\infty)$ and put $\cC_\rho := \{\xi' \in\R^2;\ \rho^{-1} \leq |\xi'| \leq \rho \}$. Then, upon applying the Plancherel theorem, we obtain
\bea
\norm{\pd_{x_3} a_j(\cdot,x_3)}_{L^2(\omega)}^2 & \leq & \norm{\pd_{x_3} \widehat{a_j}(\xi',x_3)}_{L^2(B(0,\rho^{-1}))}^2 + \rho^{-2} \int_{\R^2 \setminus B(0,\rho)} \langle \xi' \rangle^2 | \pd_{x_3} \widehat{a_j}(\xi',x_3) |^2 \d \xi' \nonumber \\
& & + \norm{\pd_{x_3} \widehat{a_j}(\xi',x_3)}_{L^2(\cC_\rho)}^2,\ x_3 \in \R,\ j=1,2. \label{t3e}
\eea
Further, as $\norm{\pd_{x_3} \widehat{a_j}(\xi',x_3)}_{L^2(B(0,\rho^{-1}))}^2 \leq |\omega| \rho^{-2} \norm{a_j}_{W^{1,\infty}(\Omega)}^2$ and
$\int_{\R^2 \setminus B(0,\rho)} \langle \xi' \rangle^2 | \pd_{x_3} \widehat{a_j}(\xi',x_3) |^2 \d \xi' \leq \norm{a_j}_{W^{1,\infty}(\Omega)}^2$, then there exists a constant $C>0$, depending only on $M$ and $\omega$, such that we have
\bel{p6} 
\norm{\pd_{x_3} \widehat{a_j}(\xi',x_3)}_{L^2(B(0,\rho^{-1}))}^2 + \rho^{-2} \int_{\R^2 \setminus B(0,\rho)} \langle \xi' \rangle^2 | \pd_{x_3} \widehat{a_j}(\xi',x_3) |^2 \d \xi' \leq \frac{M}{\rho^2}, 
\ee
according to \eqref{t3b}.
On the other hand, we derive from \eqref{t3d} that
$$
| \pd_{x_3} \widehat{a_j}(\xi',x_3) |^2 \leq C \rho^{14}(\sigma^{12} \delta^2 + \sigma^{-2\epsilon}),\ \xi'\in\mathcal C_\rho\cap B(0,\rho),\ \ x_3 \in \R,\ \sigma > \sigma_*,\ j=1,2,
$$
where $\delta := \norm{\Lambda_{A_1,q_1}-\Lambda_{A_2,q_2}}$.
Putting this and \eqref{t3e}-\eqref{p6} together, we get for every $\sigma > \sigma_*$, that
\bel{t3f}
\norm{\pd_{x_3} a_j}_{L_{x_3}^\infty(\R,L^2(\omega))}^2 \leq  C \left(\rho^{16}\sigma^{12}\delta^2+\rho^{16}\sigma^{-2\epsilon}+ \rho^{-2} \right),\ x_3 \in \R,\ j=1,2.
\ee
Now, choosing $\rho$ so large that $\rho > \sigma_*^{\epsilon \slash 8}$, we get upon taking $\sigma = \rho^{8 \slash \epsilon} > \sigma_*$ in \eqref{t3f}, that
\bel{p7}
\norm{\pd_{x_3} a_j(.,x_3)}_{L^2(\omega)}^2 \leq  C \left( \rho^{M_\epsilon} \delta^2 + \rho^{-2} \right),\ x_3 \in \R,\ j=1,2,
\ee
where $M_\epsilon := 16 + 96 \slash \epsilon$. Thus, if $\delta < \delta_0:=\sigma_*^{-\epsilon(M_\epsilon+2) \slash 16}$, then we have $\delta^{-2 \slash (M_\epsilon+2)} > \sigma_*^{\epsilon \slash 8}$, and we may apply \eqref{p7} with $\rho=\delta^{-2 \slash (M_\epsilon+2)}$, getting
\bel{p8} 
\norm{\pd_{x_3} a_j}_{L_{x_3}^\infty(\R,L^2(\omega))}^2 \leq  2 C \delta^{2 \mu_0},\ \mbox{with}\ \mu_0:=\frac{2}{M_\epsilon+2} \in (0,1),\ j=1,2.
\ee
From this and the fact, arising from \eqref{t3b}, that $\norm{\pd_{x_3} a_j}_{L_{x_3}^\infty(\R,L^2(\omega))} \leq (2 M \delta_0^{-2 \mu_0}) \delta^{2 \mu0*}$ for all $\delta \geq \delta_0$, it then follows that \eqref{p8} remains valid for every $\delta >0$. 

Finally, arguing as before with $\beta$ instead of $\pd_{x_3} A^\sharp$, we obtain in a similar way from \eqref{t3d}, that the norm $\norm{\beta}_{L_{x_3}^\infty(\R,L^2(\omega))}$ is upper bounded, up to some multiplicative constant depending only on $M$ and $\omega$, by $\delta^{\mu_0}$, and hence \eqref{t3bb} follows from this and \eqref{p8}.

\subsection{Proof of Theorem \ref{Th2}}
The proof is an adaptation of the one of \eqref{Th1}, where the adaptation is to take into account the extra information given by \eqref{BB}. 
Actually, since $A_j=(a_{1,j},a_{2,j},a_{3,*})$ and $A=(A^\sharp,0)$ with $A^\sharp=(a_1,a_2)$, by \eqref{p5}, then \eqref{BB} yields
\bel{BB1}
\norm{\pd_{x_1} a_2-\pd_{x_2} a_1}_{L_{x_3}^\infty(\R,L^2(\omega))}=\norm{\pd_{x_1} a_2-\pd_{x_2} a_1}_{L_{x_3}^\infty(-r,r;L^2(\omega))},
\ee
and
\bel{BB2}
\norm{\pd_{x_3} a_j}_{L_{x_3}^\infty(\R,L^2(\omega))}=\norm{\pd_{x_3} a_j}_{L_{x_3}^\infty(-r,r;L^2(\omega))},\ j=1,2.
\ee
More precisely, we still consider GO solutions $u_{1,\sigma}$ and $u_{2,\sigma}$, defined by \eqref{GO}-\eqref{def-b} and \eqref{def-Phi}, with $\phi_1=\pd_{x_j} \phi$, for $j=1,2,3$, and $\phi_2=\phi$, where $\phi$ is given by \eqref{def-alpha}-\eqref{7.1}. The parameter $x_3$ appearing in \eqref{def-alpha}, is taken in $(-r,r)$, and we impose 
$\sigma > (r'-r)^{-24}$, in such a way that $\phi \in\cC^\infty_0(\cD^-_R(\theta)\times(-r',r'))$. Moreover, the functions
$$
f_\sigma(t,x) = \Phi_2(2\sigma t,x)b_2(2 \sigma t,x)e^{i\sigma(x.\theta-\sigma t)}\ \mbox{and}\ g_{\sigma}=\Phi_1(2\sigma t,x)b_1(2\sigma t,x)e^{i\sigma(x.\theta-\sigma t)},\ (t,x)\in\Sigma, 
$$
lie in $H^{2,1}_0((0,T)\times\Gamma_{r'})$, and we infer from \eqref{5.5} upon arguing as the derivation of Lemma \ref{L7.1}, that
$$
\norm{\widehat{\beta}(\xi',\cdot)}_{L^\infty(-r',r')} \leq C \langle \xi' \rangle^7 \para{\sigma^6 \norm{\Lambda_{A_1,q_1,r'}-\Lambda_{A_2,q_2,r'}}+\sigma^{-\epsilon}},
$$
and that
$$
\norm{\pd_{x_3} \widehat{\beta}(\xi',\cdot)}_{L^\infty (r',r')}\leq C \langle \xi' \rangle^8 \para{\sigma^6 \norm{\Lambda_{A_1,q_1,r'}-\Lambda_{A_2,q_2,r'}} +\sigma^{-\epsilon}}.
$$
for all $\xi'\in\R^2$ and some $\epsilon>0$. Here, the constant $C$ depends only on $\omega$, $T$, $M$, $r$, $r'$ and $\epsilon$. The desired result follows from this and \eqref{BB1}-\eqref{BB2} by arguing in the same way as in the proof of Theorem \ref{Th1}.

\subsection{Proof of Theorem \ref{Th22}}

We only prove \eqref{t22a}, the derivation of \eqref{t22c} being quite similar to the one of \eqref{n-t3bb}. 
To this end, we fix $\xi' \in \R^2$, and remind that $A=(A^\sharp,0) \in \cA_0$, with $A^\sharp=(a_1,a_2)$, satisfies $\pd_{x_1} a_1 + \pd_{x_2} a_2 =0$ in $\R^2$, so we get
\beas 
\widehat{A^\sharp}(\xi',x_3) \cdot \xi' 
& = & i (2\pi)^{-1} \int_{\R^2} A^\sharp(x',x_3) \cdot \nabla_{x'} e^{-i x'\cdot\xi'} \d x' \\
& = & -i (2\pi)^{-1} \int_{\R^2} e^{-i x' \cdot \xi'} \left( \pd_{x_1} a_1 + \pd_{x_2} a_2 \right)(x',x_3) \d x' = 0,\ x_3 \in \R,
\eeas
upon integrating by parts. Thus, remembering that $\xi_\perp'=(- | \xi' |^{-1} \xi_2, | \xi' |^{-1} \xi_1)$ whenever $\xi' \neq 0$, we obtain
$$ \widehat{A^\sharp}(\xi',x_3)=(\widehat{A^\sharp}(\xi',x_3) \cdot \xi_\perp') \xi_\perp' = (-\xi_2 \widehat{a_1} + \xi_1 \widehat{a_2})(\xi',x_3) \frac{\xi_\perp'}{| \xi'|},\ x_3 \in \R, $$
and consequently
\bel{p9} 
| \xi' | \widehat{A^\sharp}(\xi',x_3) = -i \widehat{\beta}(\xi',x_3),\ \ x_3 \in \R,
\ee
from \eqref{def-beta}, the above identity being valid for $\xi'=0$ as well.
Therefore, arguing as in the derivation of \eqref{t3bb} from \eqref{7.5}, we infer from \eqref{7.6aa} and \eqref{p9} that 
\bel{t22e}
\norm{A}_{L_{x_3}^\infty(\R,L^2(\omega))}^3 \leq C \norm{\Lambda_{A_1,q_1}-\Lambda_{A_2,q_2}}^{\mu_1},
\ee
where $C>0$ and $\mu_1 \in (0,1)$ are two constants depending only on $T$, $\omega$, and $M$.

We turn now to estimating $\| q \|_{L_{x_3}^\infty(\R,H^{-1}(\omega))^3}$, where $q=q_1-q-2$.
With reference to \eqref{partition}-\eqref{a23}, we fix $x_3 \in \R$ and $k \in \{1,\ldots,N\}$, and pick a function $\phi_{*,k}$, expressed by \eqref{def-alpha}-\eqref{7.1} in the particular case where $A_\sigma^\sharp$ is uniformly zero, i.e.
\bel{def-phik}
\phi_{*,k}(y) := h \left( y' \cdot \theta + r_{x'_k} \right) e^{-\frac{i}{2} y' \cdot \xi'} \varphi_k^{1/2}(y'-(y' \cdot \theta)\theta) \alpha_{\sigma}(y_3),\ y' \in \R^2,\ y_3 \in \R.
\ee
In view of Proposition \ref{pr-4.1}, we consider a GO solution $u_{j,\sigma}$, $j=1,2$, to the magnetic Schr\"odinger equation
$(i \pd_t + \Delta_{A_j} + q_j ) u_{j,\sigma} =0$ in $Q$, described by \eqref{GO} with
\bel{p9b}
\Phi_1=\overline{\Phi_{*,k}},\  \Phi_2=\Phi_{*,k},\ \mbox{and}\ \Phi_{*,k}(t,x) := \phi_{*,k}(x'-t\theta,x_3),\ t \in \R,\ x' \in \R^2,\ x_3 \in \R.
\ee
Bearing in mind that $\nabla \cdot A=0$, we then apply \eqref{5.5} with $V=q-A \cdot (A_1+A_2)$, getting
$$
\langle q u_{2,\sigma} , u_{1,\sigma} \rangle_{L^2(Q)} = \langle A \cdot \left( (A_1+A_2) u_{2,\sigma} - 2 i \nabla u_{2,\sigma} \right) , u_{1,\sigma} \rangle_{L^2(Q)}
+ \langle (\Lambda_{A_1,q_1}-\Lambda_{A_2,q_2})f_{\sigma},g_\sigma \rangle_{L^2(\Sigma)},
$$
where $f_\sigma$ and $g_\sigma$ are given by \eqref{def-fsigma} and \eqref{def-gsigma}, respectively. Thus, we have
$$
\abs{\langle q u_{2,\sigma} , u_{1,\sigma} \rangle_{L^2(Q)}} \leq C \| A \|_{L^\infty(\Omega)^3} \| u_{1,\sigma} \|_{L^2(Q)} \| u_{2,\sigma} \|_{L^2(0,T;H^1(\Omega))}
+ \abs{\langle (\Lambda_{A_1,q_1}-\Lambda_{A_2,q_2})f_{\sigma},g_\sigma \rangle_{L^2(\Sigma)}},
$$
and consequently
\bel{p10}
\abs{\langle q u_{2,\sigma} , u_{1,\sigma} \rangle_{L^2(Q)}} \leq C  \sigma^{8 \mu} \left( \| A \|_{L^\infty(\Omega)^3} + 
\sigma^{17 \slash 3} \norm{\Lambda_{A_1,q_1}-\Lambda_{A_2,q_2}} \right) \seq{\xi'}^6,
\ee
by \eqref{r0}, \eqref{a20}, and \eqref{a23b}.

On the other hand, it follows readily from \eqref{GO} and \eqref{p9b}, that
\bel{p11}
\langle q u_{2,\sigma} , u_{1,\sigma} \rangle_{L^2(Q)} = \int_Q q(y) \Phi_{*,k}^2(2 \sigma t, y) (\overline{b_{1,\sigma}} b_{2,\sigma})(2 \sigma t, y) \d y \d t + R_{k,\sigma},
\ee
where $b_{j,\sigma}$, $j=1,2$, is given by \eqref{def-b}, and
\beas
R_{k,\sigma} & := & \int_{Q} q(y) \Phi_{*,k}(2 \sigma t,y) \left( b_{2,\sigma}(2 \sigma t,y) e^{i \sigma (y' \cdot \theta - \sigma t)} \overline{\psi_{1,\sigma}}(t,y) 
+  \psi_{2,\sigma}(t,y) \overline{b_{1,\sigma}}(2 \sigma t,y) e^{-i \sigma (y' \cdot \theta - \sigma t)} \right) \d y \d t \nonumber \\
& & + \int_{Q} q(y) (\psi_{2,\sigma} \overline{\psi_{1,\sigma}})(t,y) \d y \d t. 
\eeas
Therefore, $\abs{R_{k,\sigma}} $ is majorized by
\beas
& & \| q \|_{L^{\infty}(\Omega)} \left( \| \Phi_{*,k}(2\sigma\cdot,\cdot) \|_{L^2(Q)}  \left( \| \psi_{1,\sigma}  \|_{L^2(Q)} + \| \psi_{2,\sigma}  \|_{L^2(Q)} \right) + \| \psi_{1,\sigma}  \|_{L^2(Q)} \| \psi_{2,\sigma}  \|_{L^2(Q)} \right) \nonumber \\
& \leq & C \sigma^{-11 \slash 6} \cN_{\theta,\sigma}(\phi_{*,k})^2,
\eeas
in virtue of \eqref{def-nrm}--\eqref{4.6} and \eqref{55}, so we infer from \eqref{a23b} that
\bel{p12}
\abs{R_{k,\sigma}} \leq C  \sigma^{8 \mu - 7 \slash 6} \seq{\xi'}^6.
\ee
We turn now to examining the first term in the right hand side of \eqref{p11}. In light of \eqref{def-b}, we have
\bel{p13}
\int_Q q(y) \Phi_{*,k}^2(2 \sigma t, y) ( \overline{b_{1,\sigma}} b_{2,\sigma} ) (2 \sigma t, y) \d y \d t 
= \int_Q q(y) \Phi_{*,k}^2(2 \sigma t, y) \d y \d t + r_{k,\sigma},
\ee
with
\bel{p13a}
r_{k,\sigma} := \int_Q q(y) \Phi_{*,k}^2(2\sigma t,y) \left( e^{-i \int_0^{2 \sigma t} \theta \cdot A_\sigma^\sharp(y'-s \theta,y_3) \d s} - 1 \right) \d y \d t.
\ee
Next, as $e^{-i \int_0^{2 \sigma t} \theta \cdot A_\sigma^\sharp(y'-s \theta,y_3) \d s} - 1= -i \int_0^{2 \sigma t} \theta \cdot A_\sigma^\sharp(y'- \tau \theta,y_3) e^{-i \int_0^{\tau} \theta \cdot A_\sigma^\sharp(y'-s \theta,y_3) \d s} \d \tau$, we have 
$$ \abs{ e^{-i \int_0^{2 \sigma t} \theta \cdot A_\sigma^\sharp(y'-s \theta,y_3) \d s} - 1 }
\leq 2 \sigma T \norm{A_\sigma^\sharp}_{L^{\infty}(\R^3)^2} \leq C \sigma \norm{A}_{L^{\infty}(\Omega)^3},\ (t,y) \in Q. $$
Here we used the fact, arising from \eqref{def-a2ts} and \eqref{a40}-\eqref{def-aijsigma}, that for any $\sigma>0$, $\norm{A_\sigma^\sharp}_{L^{\infty}(\R^3)^2}$ is majorized, up to some multiplicative constant that is independent of $\sigma$, by $\norm{A^\sharp}_{L^{\infty}(\Omega)^2}$. Therefore, we infer from \eqref{t3b}, \eqref{55}, and \eqref{a23b}, that
\bel{p13b}
| r_{k,\sigma} | \leq C \sigma \norm{A}_{L^{\infty}(\R^3)^3} \norm{\Phi_{*,k}(2 \sigma \cdot, \cdot)}_{L^2(Q)}^2 
\leq C \norm{A}_{L^{\infty}(\R^3)^3} \norm{\phi_{*,k}}_{L^2(\R^3)}^2 \leq C \norm{A}_{L^{\infty}(\R^3)^3} \langle \xi' \rangle^6 \sigma^{8 \mu}.
\ee
We are left with the task of examining the integral
\bea
\int_Q q(y) \Phi_{*,k}^2(2 \sigma t, y) \d y \d t & = & \int_0^T \int_{\R^3} q(y) \phi_{*,k}^2(y'-2 \sigma t \theta, y_3) \d y' \d y_3 \d t \nonumber \\
& = & \frac{1}{2 \sigma} \int_0^{2 \sigma T} \int_{\R^3} q(y'+s \theta,y_3) \phi_{*,k}^2(y) \d y' \d y_3 \d s, \label{p14}
\eea
appearing in the right hand side of \eqref{p13}. To do that, we notice for all $\sigma > \sigma_*$, that 
$$ q(y'+s \theta,y_3) \phi_{*,k}^2(y) =0,\ s \in (-\infty,0) \cup (2 \sigma T,+\infty),\ y' \in \R^2,\ y_3 \in \R, $$
since $q$ and $\phi_{*,k}$ are supported in $B(0,R) \times \R$ and $\cD_R^-(\theta) \times \R$, respectively, and
that $\abs{y' + s \theta} > R$ whenever $y' \in \cD_R^-(\theta)$ and $s \in (-\infty,0) \cup (2 \sigma T,+\infty)$. In view of \eqref{def-P} and \eqref{p14}, this entails that
\beas
\int_Q q(y) \Phi_{*,k}^2(2 \sigma t, y) \d y \d t & = & \frac{1}{2 \sigma} \int_{\R} \int_{\R^3} q(y'+s \theta,y_3) \phi_{*,k}^2(y) \d y' \d y_3 \d s \nonumber \\
& = & \frac{1}{2 \sigma} \int_{\R^3} \cP(q)(\theta,y',y_3) \phi_{*,k}^2(y) \d y' \d y_3.
\eeas
Thus, arguing in the same way as in the derivation of \eqref{7.3}, we infer from \eqref{def-phik} that
\bel{t22h}
\abs{\widehat{q}(\xi',x_3)}
\leq C \langle \xi' \rangle^6 \left( \sigma^{20 \slash 3} \norm{\Lambda_{A_1,q_1}-\Lambda_{A_2,q_2}}+ \sigma^{8 \mu+1} \norm{A}_{L^\infty(\Omega)^3}+\sigma^{8 \mu - 1 \slash 6} \right),\ee
for every $\sigma>\sigma_*$.

The next step of the proof is to upper bound $\norm{A}_{L^\infty(\Omega)^3}$ in terms of $\norm{\Lambda_{A_1,q_1}-\Lambda_{A_2,q_2}}$. To do that, we pick $p>2$ and apply Sobolev's embedding theorem (see e.g. \cite[Corollary IX.14]{[Br]}), getting $\norm{A(\cdot,x_3)}_{L^\infty(\omega)^3}\leq C \norm{A(\cdot,x_3)}_{W^{1,p}(\omega)^3}$ for a.e. $x_3 \in \R$, where the constant $C>0$ depends only on $\omega$. Interpolating, we thus obtain that
$$ \norm{A(\cdot,x_3)}_{L^\infty(\omega)^3}\leq C \norm{A(\cdot,x_3)}_{W^{2,p}(\omega)^3}^{1 \slash 2} \norm{A(\cdot,x_3)}_{L^p(\omega)^3}^{1 \slash 2},\ x_3 \in \R. $$
This and \eqref{t22e} yield
$$
\norm{A}_{L^\infty(\Omega)^3} \leq C \norm{A}_{L_{x_3}^\infty(\R, L^p(\omega)^3)}^{1 \slash 2} \leq C\norm{A}_{L_{x_3}^\infty(\R, L^2(\omega)^3)}^{1 \slash p} \leq C\norm{\Lambda_{A_1,q_1}-\Lambda_{A_2,q_2}}^{\mu_1 \slash p},
$$
for some constant $C>0$ depending only on $\omega$, $M$ and $T$. Then, we find by substituting the right hand side of the above estimate for $\norm{A}_{L^\infty(\Omega)^3}$ in \eqref{t22h}, that
\bel{p18b}
\abs{\widehat{q}(\xi',x_3)} \leq C \seq{\xi'}^6 \left( \sigma^{20 \slash 3} \norm{\Lambda_{A_1,q_1}-\Lambda_{A_2,q_2}}+ \sigma^{8 \mu+1} \norm{\Lambda_{A_1,q_1}-\Lambda_{A_2,q_2}}^{\mu_1 \slash p}+\sigma^{8 \mu - 1 \slash 6} \right),\ \sigma > \sigma_*.
\ee 
With the notations of Subsection \ref{sec-stability-th1}, we infer from \eqref{p18b} and the estimate 
$$
\int_{\R^2 \setminus B(0,\rho)} \langle \xi' \rangle^{-2} | \widehat{q}(\xi',x_3) |^2 \d \xi' \leq \frac{M}{\rho^2},\ x_3 \in \R,
$$
which holds true for any $\rho \in (1,+\infty)$, that
\bel{p19}
\norm{q}_{L_{x_3}^\infty(\R,H^{-1}(\omega))} \leq C \left( \rho^6 \sigma^{20 \slash 3} \delta^{\mu_1 \slash p} + \sigma^{8\mu - 1 \slash 6}+ \rho^{-1} \right),\ \sigma > \sigma_*,
\ee
where $\delta = \norm{\Lambda_{A_1,q_1}-\Lambda_{A_2,q_2}} \in (0,1)$. Thus, for $\mu \in (0, 1 \slash 48)$ and $\delta \in \left( 0,\sigma_*^{-(41-48 \mu)p \slash (6 \mu_1)} \right)$, we obtain \eqref{t22a} with $\mu_2:=(1-48 \mu)\mu_1 \slash (7p (41-48\mu))$ by taking $\rho=\delta^{-\mu_2}$ and $\sigma=\delta^{-42 \mu_2 \slash (1-48 \mu)}$ in \eqref{p19}.



\end{document}